\documentclass[11pt]{article}
\usepackage[utf8]{inputenc}
\usepackage[T1]{fontenc}
\usepackage{mathtools}
\usepackage[english]{babel}
\usepackage{amsmath}
\usepackage{amsthm}
\usepackage{amsfonts}
\usepackage{amssymb}
\usepackage{graphicx}
\usepackage{xcolor}
\theoremstyle{plain}
\newtheorem{theorem}{Theorem}[section]
\newtheorem{corollary}[theorem]{Corollary}

\newtheorem{lemma}[theorem]{Lemma}
\newtheorem{observation}[theorem]{Observation}

\theoremstyle{definition}
\newtheorem{definition}[theorem]{Definition}
\newtheorem{question}[theorem]{Question}

\theoremstyle{remark}
\newtheorem*{remark}{Remark}
\usepackage[left=2cm,right=2cm,top=2cm,bottom=2cm]{geometry}
\usepackage{amssymb}
\usepackage{hyperref}
\usepackage[square,sort,comma,numbers]{natbib}
\usepackage{dsfont}
\usepackage{accents}
\usepackage{verbatim}
\usepackage[all]{xy}

\usepackage{subfig}

\usepackage{enumerate}

\title{\scshape
  New Constructions related to the Polynomial Sphere Recognition Problem}
\author{Johannes Carmesin\thanks{University of Birmingham 
} \and Lyuben Lichev\thanks{Ecole Normale Supérieure de Lyon}}

\begin{document}

\maketitle
 \begin{abstract}
We construct a simply connected $2-$complex $C$ embeddable in $3-$space such 
that for any 
embedding of $C$ in $\mathbb S^3$, any edge contraction forms a minor of the $2-$complex 
not embeddable in $3-$space. We achieve this by proving that every edge of $C$ forms a 
nontrivial knot in any of the embeddings of $C$ in $\mathbb S^3$.
 \end{abstract}

\section{Introduction}

Given a triangulation of a compact $3-$manifold, is there a polynomial
time algorithm to decide whether this $3-$manifold is homeomorphic to
the $3-$sphere? This is the Polynomial Sphere Recognition Problem.

\vspace{.3cm}

This problem has fascinated many mathematicians. Indeed, in 1992, Rubinstein proved that there is an algorithm that decides whether a given compact triangulated 3-manifold is isomorphic to the 3-sphere. This was simplified by Thompson \citep{MR1295555}.\footnote{See for example \citep{Sch11} for details on the history.} It has been shown by Schleimer \citep{Sch11} that this problem lies in NP, and by Zentner \citep{Zen16} that this problem lies in co-NP provided the generalised Riemann Hypothesis. These results suggest that there might actually be a polynomial algorithm for Sphere Recognition.

\vspace{.3cm}

The Polynomial Sphere Recognition Problem is polynomial equivalent to the
following combinatorial version (this follows for example by combining \citep{JC2} and \citep{3space5}). Given a $2-$complex $C$ whose first homology group $H_1(\mathbb{Q})$ 
over the rationals is trivial, is
there a polynomial time algorithm that decides whether $C$ can be
embedded in $3$-space (that is, the $3-$sphere or equivalently the 3-dimensional euclidean space $\mathbb{R}^3$)?

\vspace{.3cm}

In this paper, we provide new constructions that demonstrate some of the difficulties of this embedding problem. 
A naive approach towards this embedding problem is the following. Let a
$2-$complex $C$ with $H_1(\mathbb{Q})=0$ be given.
\begin{enumerate}
 \item Find an edge $e$ of $C$ such that if $C$ is embeddable, then $C/e$ is embeddable.
(For example if $e$ is not a loop and $C$ is embeddable, then $C/e$ is
embeddable. If $C$ can be embedded in such a way that there is some edge
$e'$ that is embedded as a trivial knot, then there also is an edge 
$e$ such that $C/e$ is embeddable.)
\item  By induction get an embedding of the
smaller $2-$complex $C/e$. Then use the embedding of $C/e$ to construct an
embedding of $C$.
\end{enumerate}
We will show that this strategy cannot work. More precisely we prove
the following.

\begin{theorem}\label{Thm 1}
 There is a simply connected $2-$complex $C$ embeddable in $3-$space
such that every edge $e$ forms a nontrivial knot in any embedding of $C$
and $C/e$ is not embeddable.
\end{theorem}

Our construction is quite flexible and actually can easily be modified to give an infinite set
of examples. It seems to us that the Polynomial Sphere Recognition
Problem should be difficult for constructions similar to ours. More
precisely, we offer the following open problem.

\begin{question}\label{Question 2}
Given a $2-$complex $C$ with $H_1(\mathbb{Q})=0$, is there a 
polynomial time
algorithm that decides whether there is an integer $n$ 
such that the $2-$complex $C$ can be 
obtained from  the $n \times n \times n$ cuboid complex $Z^3[n]$ by contracting a spanning tree and 
 deleting faces?
\end{question}

In a sense the 2-complexes constructed in this paper are even more obscure than embeddable 2-complexes that are contractible but not collapsible or shellable; see \citep{MR3639606} for constructions of such examples and further references in this direction. 

The remainder of this paper has the following structure. In Section \ref{sec2}
we give basic definitions and state one theorem and two lemmas, which together imply the main result, 
Theorem \ref{Thm 1}. In the following three sections, we prove these facts, one section for each. We
finish by mentioning some follow-up questions in Section \ref{sec6}.

\section{Precise statement of the results}\label{sec2}

We begin by giving a rough sketch of the construction of a $2-$complex satisfying \autoref{Thm 1}. For a $2-$complex $C$ we denote by $Sk_1(C)$ the $1-$skeleton of $C$.\par

We define the concept of \textit{cuboid graphs}. Let $n_1, n_2, n_3$ be nonnegative numbers. We 
define the sets 
\begin{equation}\label{Vc}
    V_c := \{(x,y,z)\in \mathbb Z^3\hspace{0.4em} |\hspace{0.4em} 0\leq x\leq n_1; 0\leq y\leq n_2; 
0\leq z\leq n_3\}
\end{equation}
and 
\begin{equation*}
    E_c := \{((x_1, y_1, z_1), (x_2, y_2, z_2))\hspace{0.4em}|\hspace{0.4em} |x_1-x_2| + |y_1-y_2| + 
|z_1-z_2| = 1\}.
\end{equation*}
We call the graph $G_c = (V_c, E_c)$ \textit{the cuboid graph of size $n_1\times n_2\times n_3$}.
We refer to the given embedding of the graph $G_c$ in $\mathbb R^3$ as the \textit{canonical 
embedding of the cuboid graph $G_c$}. We define \textit{the cuboid complex $C_c = (V_c, 
E_c, F_c)$ of size $n_1\times n_2\times n_3$} as the $2-$complex obtained from the cuboid graph of 
size $n_1\times n_2\times n_3$ with faces attached to every $4-$cycle. Again we refer to the 
embedding of the cuboid complex $C_c$ in $\mathbb R^3$ obtained from the canonical 
embedding of the cuboid graph $G_c$ by adding straight faces on each of its $4-$cycles as the 
\textit{canonical embedding of the cuboid complex $C_c$}, see \autoref{F1}. It induces a 
natural metric on $C_c$. This allows us in particular to refer to the vertices of $C_c$ by giving their 
cartesian coordinates in the canonical embedding of $C_c$. \par

Consider the cuboid complex $C$ of size $(2n+1)\times n\times n$ for some large $n$. We shall construct a tree $T'$ with 
edges contained in the faces of $C$ and vertices coinciding with the vertices of $C$. It will have the additional property that every fundamental cycle of the tree $T'$ seen 
as a spanning tree of the graph $T'\cup Sk_1(C)$ is knotted in a nontrivial way in every 
embedding of $C$ in $3-$space. We will use the edges of $T'$, which do not belong 
to the $1-$skeleton of $C$, to subdivide some of the faces of $C$. This will 
produce a simply connected $2-$complex $C'$. Then, by contraction of the spanning tree $T'$ of the $1-$skeleton of $C'$ we obtain the $2-$complex $C''$ with only one 
vertex and a number of loop edges. We shall show that the $2-$complex $C''$ has the following properties:

\begin{itemize}
    \item It is simply connected.
    \item It is embeddable in $3-$space in a unique way up to homeomorphism and in this embedding 
each of its edges is knotted in a nontrivial way.
    \item For every edge $e$ of $C''$, the $2-$complex $C''/e$ obtained by 
contraction of $e$ in $C''$ does not embed in $3-$space.
\end{itemize}

To formalise these ideas, we begin with a definition.

\begin{definition}\label{Def 1}
\normalfont
Let $C_1$ be a $2-$complex with an embedding $\iota_1$ in $3-$space. Let $T_1$ be a 
spanning tree of the $1-$skeleton of $C_1$. The tree $T_1$ is \textit{entangled with 
respect to $\iota_1$} if, for any edge $e_1$ of $C_1$ outside $T_1$, the fundamental cycle 
of $e_1$ is a nontrivial knot in the embedding $\iota_1$. Moreover, if $T_1$ is entangled with 
respect to every embedding of the $2-$complex $C_1$ in $3-$space, we say that $T_1$ is 
\textit{entangled}.
\end{definition}

Let $C = (V, E, F)$ be the cuboid complex of size $(2n+1)\times n\times n$ for $n$ 
at least 20. The last condition might seem artificial. It is a sufficient condition for the existence of a special type of path constructed later in
the proof of Lemma \ref{spine}. If it confuses the reader, one might consider $n$ large enough till 
the end of the paper. 

\begin{theorem}\label{Thm 2}
There exists a simply connected $2-$complex $C' = (V', E', F')$ with an entangled spanning 
tree $T'$ of the $1-$skeleton of $C'$ with the following properties:
\begin{itemize}
    \item $C'$ is constructed from $C$ by subdividing some of the faces of the 
$2-$complex $C$ of size four into two faces of size three. We refer to 
the edges participating in these subdivisions as \textit{diagonal edges}.
    \item $T'$ contains all diagonal edges. Moreover, every fundamental cycle of $T'$ in the $1-$skeleton of $C'$ contains three consecutive diagonal edges in the same 
direction (i.e. collinear in the embedding of $C'$ induced by the canonical embedding of 
$C$). 
\end{itemize}
\end{theorem}

We remark that, indeed, adding the appropriate subdividing edges as line segments contained in the faces of 
the $2-$complex $C$ within its canonical embedding induces an embedding of $C'$. 
We call this induced embedding the \textit{canonical embedding of $C'$}.

Having a $2-$complex $C'$ and an entangled spanning tree $T'$ of the 
$1-$skeleton of $C'$ satisfying the conditions of \autoref{Thm 2} we construct our main 
object of interest as follows. Let the $2-$complex $C''$ be obtained after contraction of the tree $T'$ in $C'$ i.e. $C'' = C'/T'$.\par

We now make a couple of observations.

\begin{observation}\label{Ob1}
In an embedding $\iota$ of a $2-$complex $C$ in $3-$space every edge that is not a loop can be geometrically contracted within the embedding.
\end{observation}

\begin{proof}
Let $e$ be an edge of $C$ that is not a loop. Consider a tubular neighbourhood $D_e$ of $\iota(e)$ in $\mathbb S^3$ such that $D_e\cap \iota$ is connected in $D_e$. Now we can contract $\iota(e)$ within $D_e$. This is equivalent to contracting $e$ within $\iota$ while keeping $\iota \cap (\mathbb S^3\backslash D_e)$ fixed.
\end{proof}

Thus the $2-$complex $C''$ will be embeddable in $3-$space. 

\begin{observation}\label{Ob2}
Contraction of any edge in a simply connected $2-$complex (even one forming a loop) produces a 
simply connected $2-$complex.
\end{observation}

\begin{proof}
Contraction of edges in a $2-$complex does not modify its topology. In particular, the property of being simply connected is preserved under edge contraction.
\end{proof}

Thus by construction the $2-$complex $C''$ will be simply connected.\\

The next lemma is proved in Section \ref{Section4}.

\begin{lemma}\label{sec4lemma}
Every embedding of the $2-$complex $C''$ in $3-$space is obtained from an embedding of the $2-$complex $C'$ by contracting the spanning tree $T'$.
\end{lemma}

\begin{remark}
Notice that Observation \ref{Ob1} ensures that contraction of $T'$ can be done within any given embedding of $C'$ in $3-$space.
\end{remark}

The next lemma is proved in Section \ref{Section5}.

\begin{lemma}\label{sec5lemma}
For every edge $e''$ of $C''$ the $2-$complex $C''/e''$ does not embed in $3-$space. 
\end{lemma}

Admitting the results of Section \ref{sec2} we stated so far, we prove \autoref{Thm 1}.

\begin{proof}[Proof of \autoref{Thm 1} assuming \autoref{Thm 2}, Lemma \ref{sec4lemma} and Lemma \ref{sec5lemma}.]
We show that $C''$ satisfies \autoref{Thm 1}. Firstly, we have by Observation \ref{Ob2} that $C''$ is a simply connected $2-$complex. Secondly, by Observation \ref{Ob1} we have that $C''$ is embeddable in $3-$space, but by Lemma \ref{sec5lemma} for every edge $e''$ of $C''$ we have that $C''/e''$ is not embeddable in $3-$space. Finally, let $\iota''$ be an embedding of $C''$ in $3-$space and let $e''$ be an edge of $C''$. The edge $e''$ corresponds to an edge $e'$ of $C'$ not in $T'$. By Lemma \ref{sec4lemma} $\iota''$ originates from an embedding $\iota'$ of the $2-$complex $C'$. But by \autoref{Thm 2} we have that the tree $T'$ is entangled, so the fundamental cycle of $e'$ in the embedding of $T'$ induced by $\iota'$ forms a nontrivial knot. As contracting $T'$ within $\iota'$ preserves the knot type of its fundamental cycles, $\iota''(e'')$ is a nontrivial knot. Thus every edge of $C''$ forms a nontrivial knot in each of the embeddings of $C''$ in $3-$space, which finishes the proof of \autoref{Thm 1}.
\end{proof}

\vspace{.3cm}

We now turn to several important definitions.

\begin{definition}\label{Def 2.2}
\normalfont A 
\textit{connected sum} of two knots is an operation defined on their disjoint union as follows. See \autoref{Steps}
\begin{enumerate}
    \item Consider a planar projection of each knot and suppose these projections are disjoint.
    \item Find a rectangle in the plane where one pair of opposite sides are arcs along each knot, 
but is otherwise disjoint from the knots and so that the arcs of the knots on the sides of the 
rectangle are oriented around the boundary of the rectangle in the same direction.
    \item Now join the two knots together by deleting these arcs from the knots and adding the arcs 
that form the other pair of sides of the rectangle.
\end{enumerate}
\end{definition}

We remark that the definition of connected sum of two knots is independent of the choice of planar 
projection in the first step and of the choice of rectangle in the second step in the sense that the 
knot type of the resulting knot is uniquely defined. By abuse of language we will often call connected sum of the knots $K$ and $K'$ the knot obtained by performing the operation of connected sum on $K$ and $K'$. This new knot will be denoted $K\# K'$ in the sequel. \par
In the proof of Lemma \ref{L 3.6} we rely on the following well-known fact.\par

\begin{figure}
\centering
\includegraphics[scale=0.2]{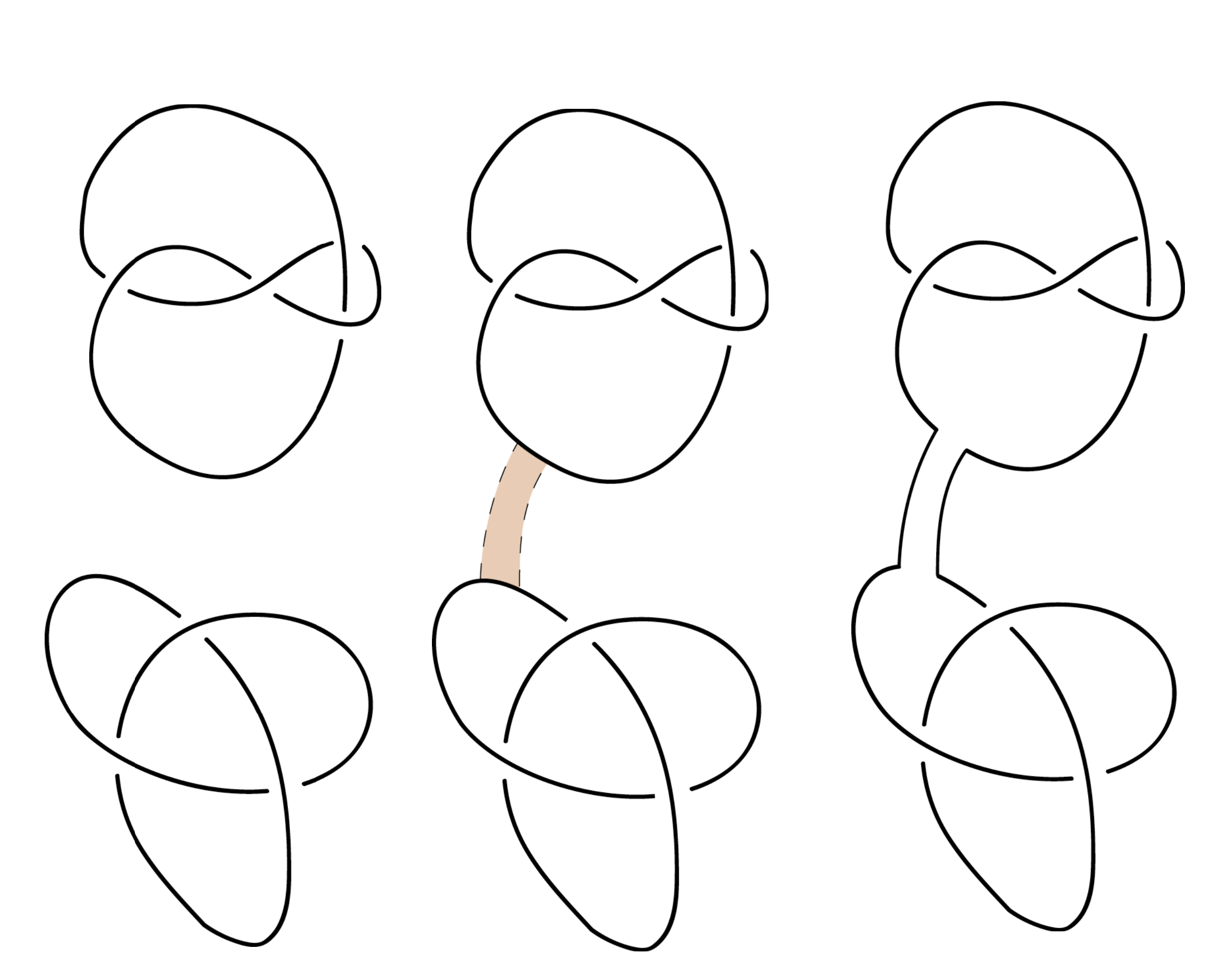}
\caption{The operation of connected sum between two disjoint knots. The figure illustrates the three 
steps in the definition. Source: Wikipedia.}
\label{Steps}
\end{figure}

\begin{lemma}\label{L 2.6}(\citep{ZH}, \citep{Kos})
A connected sum of two knots is trivial if and only if each of them is trivial.
\end{lemma}

Let $\phi: \mathbb S^1 \longrightarrow \mathbb S^3$ be an embedding of the unit circle in $3-$space. 
The knot $\phi(\mathbb S^1)$ is \textit{tame} if there exists an extension of $\phi$ to an embedding 
of the solid torus $\mathbb S^1 \times D^2$ into the $3-$sphere. Here, $D^2$ is the closed unit 
disk. We call the image of this extension into the $3-$sphere \textit{thickening} of the knot. We 
remark that the image of a tame knot by a homeomorphism of the $3-$sphere is again a tame knot. In 
this paper we consider only piecewise linear knots, which are tame.\par

The following definitions can be found in \citep{JC1}. The graph $G$ is \textit{$k-$connected} if it has at least $k+1$ vertices and
for every set of $k-1$ vertices $\{v_1, v_2, \dots v_{k-1}\}$ of $G$, the graph obtained from $G$ by 
deleting the vertices $v_1, v_2, \dots v_{k-1}$ is connected. A \textit{rotation system of the graph 
$G$} is a family $(\sigma_v)_{v\in V(G)}$, where $\sigma_v$ is a cyclic orientation of the edges 
incident with the vertex $v$ in $G$. For the rotation system $(\sigma_v)_{v\in V(G)}$ of the graph 
$G$ and for every vertex $v$ in $G$ we call $\sigma_v$ the \textit{rotator} of $v$. A rotation 
system of a graph $G$ is \textit{planar} if it induces a planar embedding of $G$.\par

Let $G' = (V', E')$ and $G'' = (V'', E'')$ be two disjoint graphs. Let $v'$ and $v''$ be vertices of 
equal degrees in $G'$ and $G''$ with neighbours $(u'_1, \dots, u'_k)$ and $(u''_1, \dots, u''_k)$ 
respectively. We define a bijection $\varphi$ between $(u'_1, \dots, u'_k)$ and $(u''_1, \dots, 
u''_k)$ by 
\begin{equation*}
\forall i\leq k,\hspace{0.4em} \varphi(u'_i) = u''_i.
\end{equation*} 
The \textit{vertex sum of $G'$ and $G''$ at $v'$ and $v''$ over $\varphi$} is a graph $G$ obtained from 
the disjoint union of $G'$ and $G''$ by deleting $v'$ from $G'$ and $v''$ from $G''$ and adding the 
edges $(u'_i, u''_i)_{1\leq i\leq k}$. We sometimes abuse the term vertex sum to refer to the 
operation itself. We say that an edge $e'$ of the graph $G'$ is \textit{inherited} by the vertex sum $G$ from the graph $G'$ if its two endvertices are both different from $v'$. A vertex $v'$ of the graph $G'$ is \textit{inherited} by the vertex sum $G$ from the graph $G'$ if it is different from $v'$. Thus $e'$ (respectively $v'$) can be viewed both as an edge (respectively vertex) of $G$ and as an edge (respectively vertex) of $G'$. See \autoref{VertexSum}.\par 
Moreover, consider graphs $G'$ and $G''$ with rotation systems $(\sigma_u')_{u'\in V'}$ and 
$(\sigma''_{u''})_{u''\in V''}$ and vertices $v'$ in $G'$ and $v''$ in $G''$ with rotators 
$\sigma_{v'} = (u'_1, \dots, u'_k)$ and $\sigma_{v''} = (u''_1, \dots, u''_k)$ respectively. There 
is a bijection $\phi$ between the rotators of $v'$ and $v''$ defined up to the choice of a vertex 
from $(u''_i)_{1\leq i\leq k}$ for $\phi(u'_1)$. Once this $u''_j$ is determined, we 
construct the edges $(u'_iu''_{(i+j-1 \mod k)})_{1\leq i\leq k}$. This gives the vertex sum of $G'$ 
and $G''$ at $v'$ and $v''$ over $\phi$.

\begin{figure}
\centering
\includegraphics[scale=0.7]{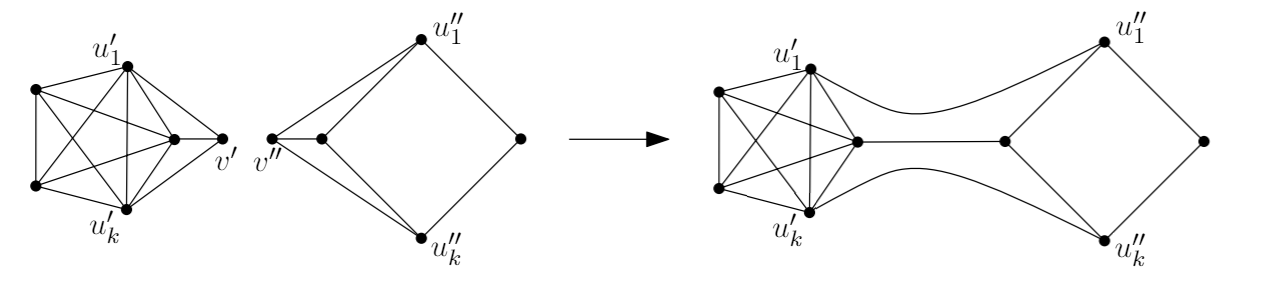}
\caption{Vertex sum of $G'$ and $G''$ at $v'$ and $v''$ respectively. The edge $u'_1u'_k$ is inherited by the vertex sum from $G'$.}
\label{VertexSum}
\end{figure}

We now give a couple of definitions for $2-$complexes. Let $C _1 = (V_1, E_1, F_1)$ be a 
$2-$complex and let $v$ be a vertex in $C_1$. The \textit{link graph $L_v(C _1)$ 
at $v$ in $C_1$} is the incidence graph between edges and faces incident with $v$ in 
$C_1$. See \autoref{Link}. A \textit{rotation system of the $2-$complex $C _1$} is 
a family $(\sigma_e)_{e\in E_1}$, where $\sigma_e$ is a cyclic orientation of the faces incident 
with the edge $e$ of $C_1$. A rotation system of a $2-$complex $C_1$ induces a rotation 
system on each of its link graphs $L_v(C _1)$ by restriction to the edges incident with the 
vertex $v$. A rotation system of a $2-$complex is \textit{planar} if all of the induced rotation systems
on the link graphs at the vertices of $C_1$ are planar.\par

\begin{figure}
\centering
\includegraphics[scale=0.7]{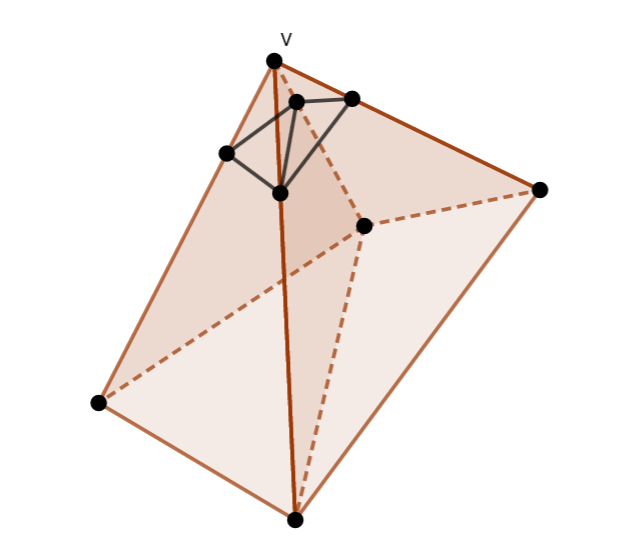}
\caption{The link graph at the vertex $v$ is given in black.}
\label{Link}
\end{figure}

\section{Proof of \autoref{Thm 2}}

In this section we work with the cuboid complex $C$ of size $(2n+1)\times n\times n$ for $n\geq 20$. From this point up to Lemma \ref{L 3.6} included we suppress the map from the cuboid complex $C$ to its canonical embedding from the notation. We define the subcomplex $C_{[a,b]}$ of $C$ as the intersection of $C$ with the strip $\{a\leq x\leq b\}\subset \mathbb R^3$. If $a=b$, we write $C_{[a]}$ instead of $C_{[a,a]}$. \par
As the
$1-$skeleton of $C$ is a connected bipartite graph, it has a unique proper vertex $2-$colouring in black 
and white (up to exchanging the two colours). This colouring is depicted in \autoref{F1}. We fix this colouring.\\

\textbf{Sketch of a construction of $T'$ and $C'$ from Theorem \ref{Thm 2}.}
We define an \textit{overhand path} as a piecewise linear path in $3-$space such that by connecting its endvertices via a line segment we obtain a copy of the trefoil knot. We construct a path called \textit{spine} contained in the faces of the $2-$complex $C$ that consists roughly of two consecutive overhand paths. See \autoref{Kn}.\par
The spine contains two of the edges of $C$ serving as transitions between vertices of 
different colours and all remaining edges in the spine are diagonal edges (these diagonal edges are not actually edges of $C$ but straight-line segments subdividing of face of $C$ of size four into two triangles. The endvertices of a diagonal edge always have the same colour.). More precisely, the spine starts 
with a short path $P_1$ that we later call \textit{starting segment} of three white and two black 
vertices. We call its last black vertex $A$. See \autoref{F2}. The vertex $A$ also serves as a starting vertex of the first overhand path $P_2$, which is entirely contained in the subcomplex $C_{[n+2, 2n+1]}$ and uses only diagonal edges. The ending vertex $B$ of $P_2$ is connected via 
a path $P_3$ of three diagonal edges in the same direction to a black vertex $A'$. The vertex $A'$ serves as a starting vertex of the second overhand path $P_4$. The path $P_4$ uses only diagonal edges and is
entirely contained in the subcomplex $C_{[0, n-1]}$ of $C$. Finally, the 
ending vertex $B'$ of $P_4$ is the beginning of a short path $P_5$ of two black and 
three white vertices. We later call $P_5$ \textit{ending segment}. Visually $P_5$ is obtained from the starting segment $P_1$ by performing central symmetry. The spine is obtained by joining together the paths $P_1, P_2, 
P_3, P_4$ and $P_5$ in this order. See \autoref{Kn}. Moreover, we construct the spine $P$ so that no 
two non-consecutive vertices in $P$ are at distance 1 for the euclidean metric of $\mathbb R^3$.\par

Recall that diagonal edges in $C$ subdivide faces of size four of $C$ into two 
faces of size three. By adding further diagonal edges, we extend the spine $P$ to a tree
$T'$, whose vertex set is the set of vertices of $C$. Thus the tree $T'$ is a spanning tree of the graph $Sk_1(C)\cup T'$ obtained from the 1-skeleton of $C$ by adding the edges of $T'$. We will show that we can choose the diagonal edges in the  previous step so that for any edge $e$ of $C$ and not in $T'$, the fundamental cycle of $e$ in $T'$ contains either the path $P_2$ from $A$ 
to $B$ or the path $P_4$ from $A'$ to $B'$. Both these paths have the structure of an overhand path. 

Finally, we obtain the $2-$complex $C'$ from $C$ by subdividing the faces of $C$ that contain diagonal edges of $T'$ by those diagonal edges. This $2-$complex $C'$ is simply connected. This completes the informal description of the construction of $C'$ and $T'$.\\

\begin{figure}
\centering
\includegraphics[scale=0.9]{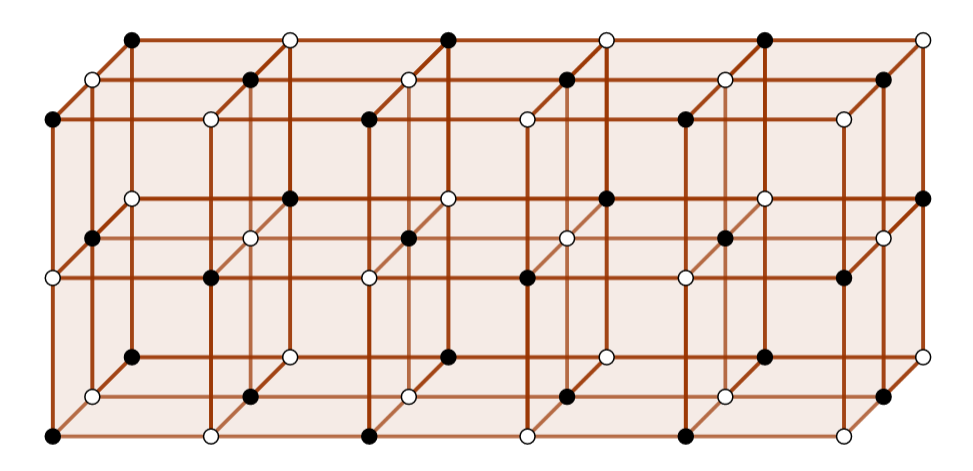}
\caption{The canonical embedding of the cuboid complex of size $5\times 2\times 2$ together with the proper $2-$colouring of its vertices.}
\label{F1}
\end{figure}

{\bf Formal construction of $T'$ and $C'$.} 
We do it in three steps.\\

We call a piecewise linear path contained in $C$ \textit{facial 
path} if:
\begin{itemize}
    \item It does not meet the edges of $C$ in points different from their endvertices.
    \item It does not contain any vertex of $V(C)$ more than once.
    \item Every pair of consecutive vertices of $C$ with respect to the order induced by 
the path are not neighbours in the $1-$skeleton of $C$.
    \item The parts of the path between two consecutive vertices of $C$ are embedded as single line segments contained in single faces of $C$.
\end{itemize} 
Informally a facial path is a path of diagonal edges without repetition of vertices. See 
\autoref{FH}. We remark that the diagonal edges are not edges of $C$.\par 
We recall that an \textit{overhand path} is a piecewise linear path in $3-$space such 
that after joining its endvertices by a line segment we obtain a copy of the trefoil knot.\par

The next definition is technical. Informally, \lq doubly knotted paths\rq \ look like the black subpath between 
$A$ and $B'$ in \autoref{Kn} up to rescaling.
A facial path $P$ in $C$ is a \textit{doubly knotted} if there exists vertices $A$, $B$, $A'$ and $B'$ appearing in that order on the facial path satisfying all of the following.
\begin{enumerate}
    \item the subpaths $APB$ and $A'PB'$ are disjoint and have each the structure of overhand paths;
    \item each of the subpaths $APB$ and $A'PB'$ contains three consecutive diagonal edges in the same direction
i.e. collinear in (the canonical embedding of) $C$;
\item the intersection 
of a facial path with the half-space $n+2 \leq x$ is exactly the subpath $APB$;
\item the 
intersection of a facial path with the half-space $x\leq n-1$ is exactly its subpath $A'PB'$;
\item the intersection of a facial path with the strip $n-1< x < n+2$ is exactly the subpath $BPA'$ (this time without the endvertices $A'$ and $B$ themselves). 
\end{enumerate}

A \textit{starting segment} is a piecewise linear path made of three diagonal edges and one edge of $C$ joining vertices with coordinates $((x,y,z), (x+1,y+1,z), (x+1,y,z+1), 
(x+2,y,z+1), (x+3,y+1,z+1))$ in this order. We call the vertex $(x,y,z)$ \textit{starting vertex} of the starting segment. We remark that every starting segment is characterised by its starting vertex. See \autoref{F2}. Likewise, an \textit{ending segment} is a piecewise linear path made of three diagonal edges and one edge of $C$ joining vertices 
with coordinates $((x,y,z), (x+1,y+1,z), (x+2,y+1,z), (x+2,y,z+1), (x+3,y+1,z+1))$. Again we call the vertex $(x,y,z)$, which indeed charachterises the ending segment, \textit{starting vertex} of the ending segment.\par
We remark that starting segments, ending segments and doubly knotted paths are not defined up to rotation but actually as explicit sets of vertices and edges (either diagonal edges or edges of $C$). Hence their concatenation is only possible in a unique way. This allows us to define a \textit{spine} as a path constructed by concatenating consecutively a starting segment, a doubly knotted path and an ending segment in this 
order. Spines have roughly the form of the path given in \autoref{Kn}. We call the doubly knotted path 
participating in a spine \textit{basis} of this spine.

\begin{lemma}\label{spine}
There exists a spine.
\end{lemma}

\begin{proof}[Proof of Lemma \ref{spine}.]
We construct a spine $P$ as a concatenation of five shorter paths. A rough sketch illustrating the construction could be found in 
\autoref{Kn}. Recall that $C$ is of size $(2n+1)\times n\times n$ for $n\geq 20$.\par 
Let us colour the vertex with coordinates $(n-1, 0, 0)$ in white. This uniquely defines the (proper) $2-$colouring of the vertices of $C$ in black and white. The construction of the spine begins with a starting segment $P_1$ with starting vertex $(n-1, 0, 0)$. We denote by $O$ the vertex with coordinates $(n, 0, 1)$ (which is white as it has even distance from the vertex $(n-1, 0, 0)$) and by $A$
the vertex with coordinates $(n+2, 1, 1)$ (which is black). See \autoref{F2}.

\begin{figure}
\centering
\includegraphics[scale=0.55]{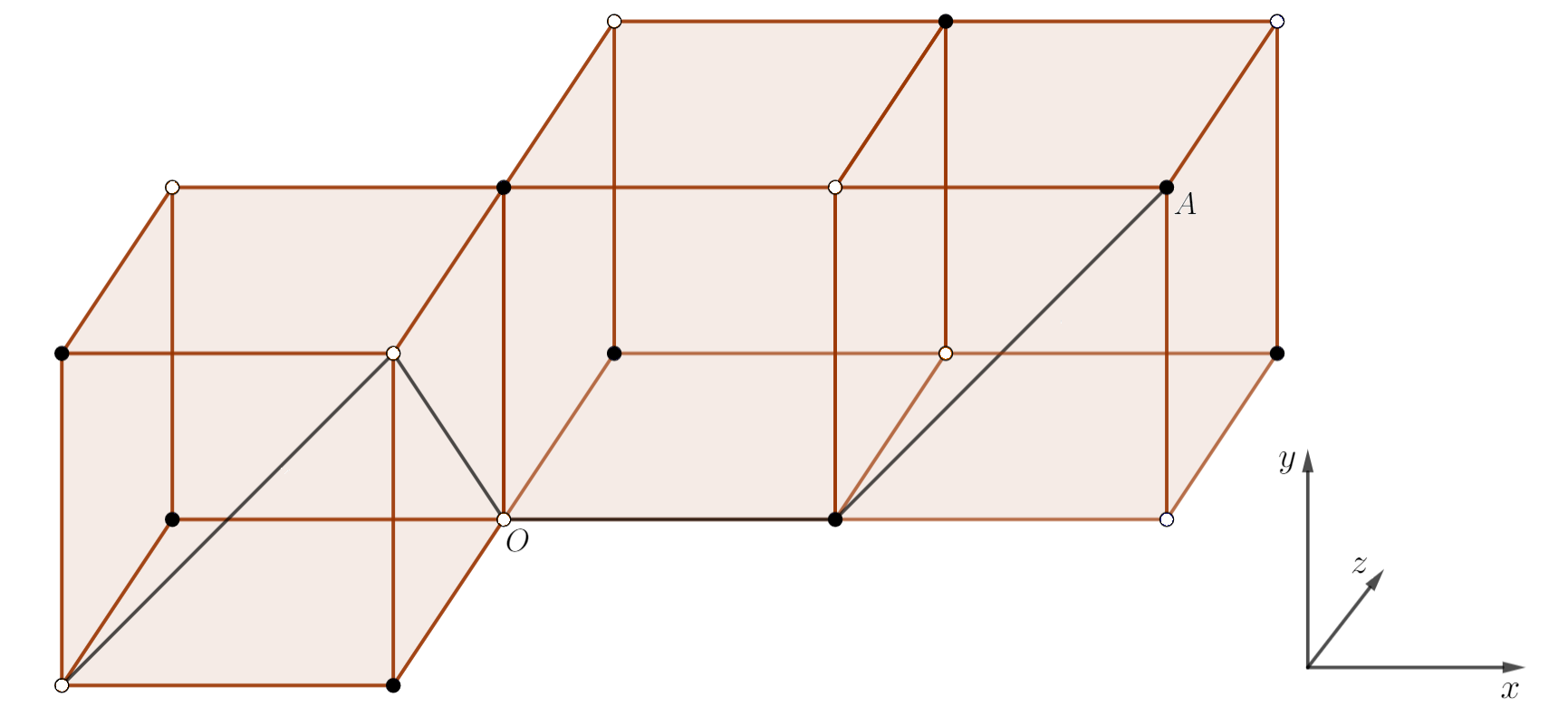}
\caption{The starting segment $P_1$ is given by concatenating the four edges coloured in black (these are three diagonal edges and one edge of $C$).}
\label{F2}
\end{figure}

\begin{figure}
\centering
\includegraphics[scale=0.50]{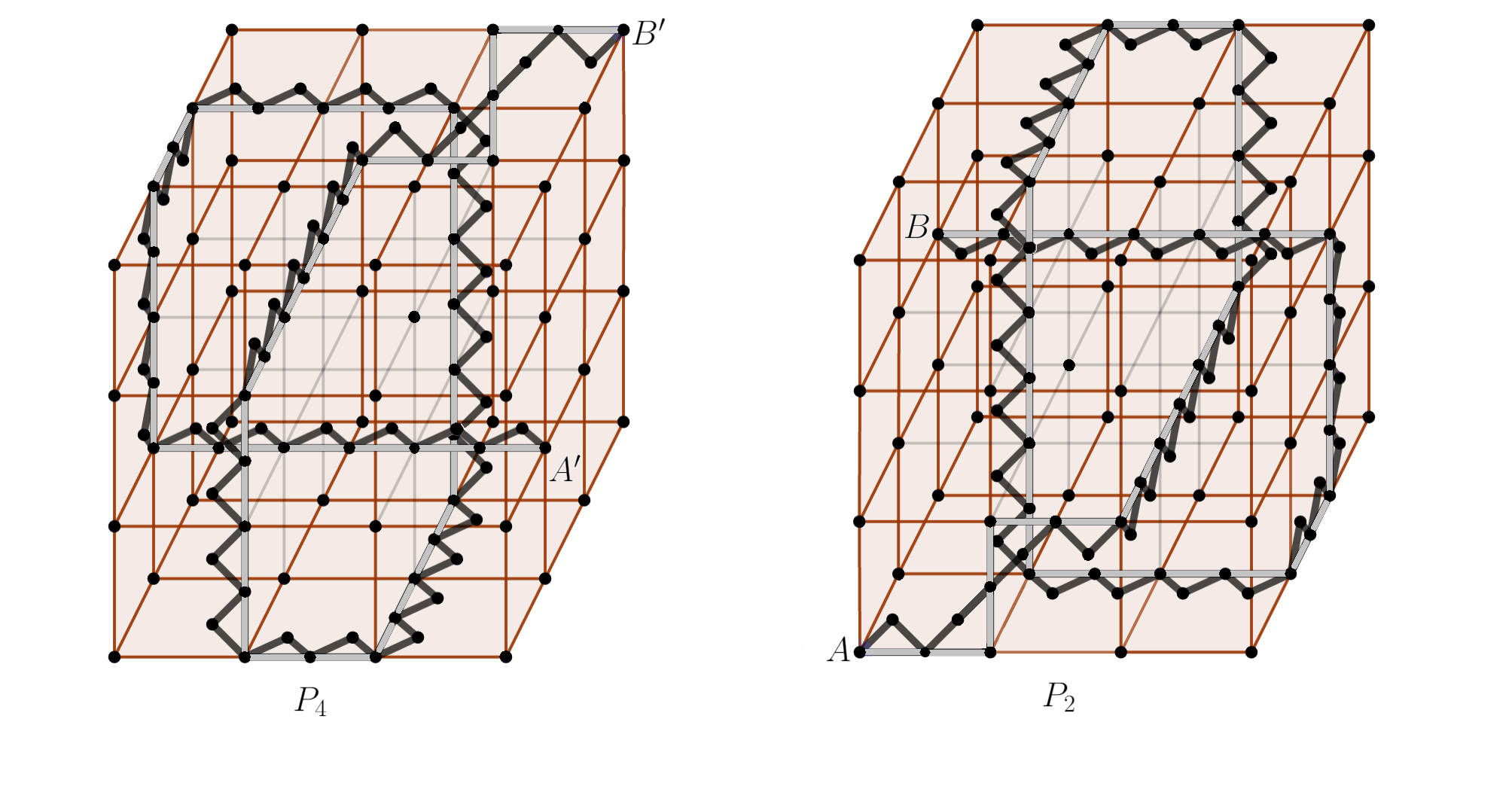}
\caption{The subpath $P_4$ of the spine $P$ on the left and and the subpath $P_2$ on the right. Both $P_2$ and $P_4$ are overhand facial paths contained in two cuboid subcomplexes of $C$ of size $12\times 12\times 12$. Only a few vertices necessary for the construction of the paths are depicted.}
\label{FH}
\end{figure}

Next, denote by $B$ the vertex with coordinates $(n+2, 9,9)$. We build an overhand facial path $P_2$ of black vertices of abscissas (i.e. first coordinates) at least $n+3$ except its first vertex $A$ 
and its last vertex $B$, which have abscissas exactly $n+2$. We define $P_2$ to be the facial path given in the right part of \autoref{FH}, which is embedded in the faces of the cuboid subcomplex of $C$ of size $12\times 12\times 12$ with $A$ being its closest vertex to the origin\footnote{Formally the path $P_2$ is given by the fact that it is a facial path approximating (i.e. staying at distance at most 1 from) the following piecewise linear path contained in the $1-$skeleton of $C$:
\begin{align*}
    A = & (n+2, 1, 1), (n+6, 1, 1), (n+6, 5, 1), (n+10, 5, 1), (n+10, 5, 13), (n+10, 13, 13), (n+6, 13, 13), (n+6, 13, 5),\\ & (n+6, 1, 5), (n+14, 1, 5), (n+14, 1, 9), (n+14, 9, 9), (n+2, 9, 9) = B.
\end{align*} Although such approximating facial path is not unique, any choice of such path is adapted for our purposes. In this proof, one particular choice of $P_2$ is made for concreteness.}.
We remark that in the figure only vertices 
important for the construction of the path are depicted.\par 

Denote by $A'$ the vertex with coordinates $(n-1, 9, 6)$. We construct a facial path $P_3$ consisting of three diagonal edges in the same direction connecting the black vertex $B$ to the black vertex $A'$.\par

Next, let $B'$ be the vertex with coordinates $(n-1, 17, 14)$. We build an overhand facial path $P_4$ of black vertices
of abscissas at most $n-2$ except the first vertex $A'$ 
and the last vertex $B'$, which have abscissas exactly $n-1$. We define $P_4$ to be the facial path given in the left part of \autoref{FH}, which is embedded in the faces of the cuboid subcomplex of $C$ of size $12\times 12\times 12$ with $B'$ being its farthest vertex to the origin\footnote{Like in the case of $P_2$, the facial path $P_4$ is formally given by an approximation of (i.e. path staying at distance at most 1 from) the following piecewise linear path contained in the $1-$skeleton of $C$:
\begin{align*}
    A' = & (n-1, 9, 6), (n-13, 9, 6), (n-13, 17, 6), (n-13, 17, 10), (n-5, 17, 10), (n-5, 5, 10), (n-5, 5, 2), (n-9, 5, 2),\\ & (n-9, 13, 2), (n-9, 13, 14), (n-5, 13, 14), (n-5, 17, 14), (n-1, 17, 14) = B'.
\end{align*} Again, despite the fact that any such approximating facial path is adapted for our purposes, in this proof we stick to a particular choice of $P_4$.}. Once again only vertices important for the construction of the path are depicted.\par 
We call $P_{[2,4]}$ the facial path between $A$ and $B'$ constructed by concatenating $P_2, P_3$ and 
$P_4$ in this order. It is doubly knotted by construction and will serve as basis of the spine $P$.\par 

Next, construct an ending segment $P_5$ with starting vertex $B'$. This is possible as $n\geq 20$. Let $O'$ be the first white vertex in $P_5$ with coordinates $(n + 1, 18, 14)$. Visually $P_5$ is obtained after central symmetry of \autoref{F2}.\par 

The spine $P$ is finally obtained by concatenating the starting segment $P_1$, the doubly knotted path $P_{[2,4]}$ and the ending segment $P_5$ in this order.
\end{proof}

We introduce the context of our next lemma. Fix three positive integers $x_1,y_1,z_1$ and let 
$C_1 = (V_1, E_1, F_1)$ be the cuboid complex of size $x_1\times y_1\times z_1$. Its 
$1-$skeleton is a connected bipartite graph so it admits a unique $2-$colouring up to exchanging the 
two colours. We fix this colouring in black and white, where vertex $(0, 0, 0)$ is white for concreteness, see \autoref{F1}. Moreover, from now up to the end of Observation \ref{ob 3.3} we suppress the map from the cuboid complex $C_1$ to its canonical embedding from the notation just like we did with the cuboid complex $C$.\par 

Let $G_b = (V_{1,b}, E(G_b))$ be a forest, where $V_{1,b}$ is the set of black vertices of $C_1$ and $E(G_b)$ is a subset of the set $E_{1,b}$ of diagonal edges with two black endvertices in 
$C_1$. Likewise let $V_{1,w}$ be the set of white vertices of $C_1$ and $E_{1,w}$ be the set of 
diagonal edges with two white endvertices in $C_1$. Finally, let $I_1\subset E_{1,w}$ be the 
set of diagonal edges with two white endvertices intersecting an edge of $G_b$ in an internal 
point.

\begin{lemma}\label{connected}
The graph $(V_{1,w}, E_{1,w}\backslash I_1)$ is connected.
\end{lemma}
\begin{proof}
We argue by contradiction. Suppose that the graph $(V_{1,w}, E_{1,w}\backslash I_1)$ is not connected. This means that there is a cuboid subcomplex $K$ of $C_1$ of size $1\times 1\times 1$ (i.e. a unit cube) with white vertices not all 
in the same connected component of $(V_{1,w}, E_{1,w}\backslash I_1)$. Suppose that the vertex of $K$ closest to $(0, 0, 0)$ is white and let $(w_1, w_2, w_3)$ be its 
coordinates (the case when this vertex is black is treated analogously). Then, if the connected component of $(w_1, w_2, w_3)$ in $(V_{1,w}, E_{1,w}\backslash 
I_1)$ contains none of $(w_1 + 1, w_2 + 1, w_3), (w_1 + 1, w_2, w_3 + 1)$ and $(w_1, w_2 + 1, w_3 + 
1)$, then the black diagonal edges $(w_1 + 1, w_2, w_3)(w_1, w_2 + 1, w_3)$, $(w_1 + 1, w_2, 
w_3)(w_1, w_2, w_3 + 1)$ and $(w_1, w_2 + 1, w_3)(w_1, w_2, w_3 + 1)$ are present in $E(G_b)$. See 
the left part of \autoref{SE}. This contradicts the fact that $G_b$ is a forest.\par
If the conected component of $(w_1, w_2, w_3)$ in $(V_{1,w}, E_{1,w}\backslash I_1)$ contains 
exactly one of the white vertices $(w_1 + 1, w_2 + 1, w_3), (w_1 + 1, w_2, w_3 + 1)$ and $(w_1, w_2 
+ 1, w_3 + 1)$, we may assume by symmetry that this is the vertex $(w_1 + 1, w_2 + 1, w_3)$. Then 
the black diagonal edges $(w_1, w_2 + 1, w_3)(w_1 + 1, w_2 + 1, w_3 + 1)$, $(w_1 + 1, w_2 + 1, w_3 + 
1)(w_1 + 1, w_2, w_3)$, $(w_1 + 1, w_2, w_3)(w_1, w_2, w_3 + 1)$ and $(w_1, w_2, w_3 + 1)(w_1, w_2 + 
1, w_3)$ are present in $E(G_b)$. See the right part of \autoref{SE}. Again, this contradicts the 
fact that $G_b$ is a forest.\par
It follows that the connected component of $(w_1, w_2, w_3)$ in $(V_{1,w}, E_{1,w}\backslash I_1)$ 
contains at least two of the other three white vertices in $K$. By symmetry 
this holds for every white vertex in $K$, which contradicts our initial assumption that not all white vertices of $K$ are in the same connected component of the graph $(V_{1,w}, E_{1,w}\backslash I_1)$. This 
proves the lemma.
\end{proof}

\begin{figure}
\centering
\includegraphics[scale=0.7]{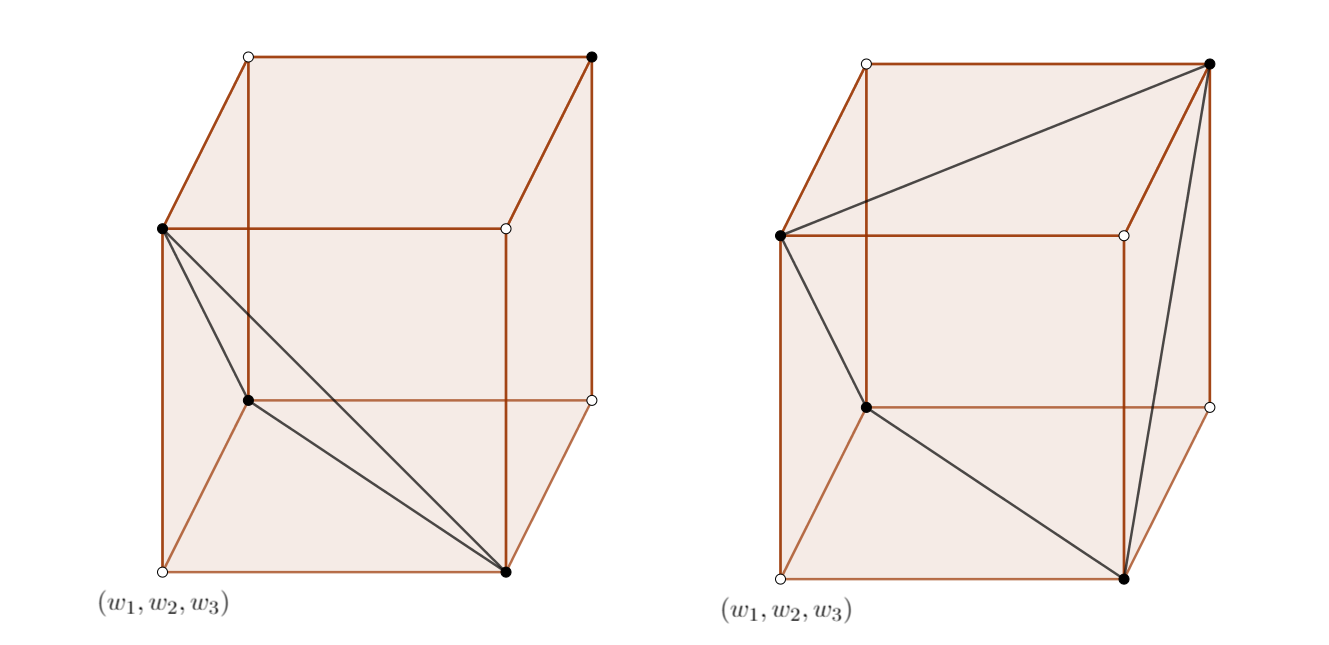}
\caption{On the left, the case when the connected component of $(w_1, w_2, w_3)$ in $(V_{1,w}, 
E_{1,w}\backslash I_1)$ contains no other white vertex in $K$. On the right, 
the case when the connected component of $(w_1, w_2, w_3)$ in $(V_{1,w}, E_{1,w}\backslash I_1)$ 
contains only the white vertex with coordinates $(w_1 + 1, w_2 + 1, w_3)$ in $K$.}
\label{SE}
\end{figure}

\begin{observation}\label{ob 3.3}
Every forest that is a subgraph of a connected graph $G$ can be extended to a spanning tree of $G$.
\end{observation}
\begin{proof}
The spanning tree can be obtained from the forest by adding edges one by one in such a way that no cycle is formed until this is possible.
\end{proof}

From now up to Lemma \ref{spanning_tree_extend} included we denote by $V_w$ or $V_b$ the set of white or black vertices of $C$, respectively. By $E_w$ or $E_b$ we denote the set of diagonal edges with two white or black endvertices in $C$, respectively.\par 
We construct a graph $G_{b, centre}$ as follows. \begin{enumerate}
    \item Consider the restriction $\Tilde{G}_{b, centre}$ of the graph $(V_b, E_b)$ to the vertex set of $C_{[n, n+1]}$.
    \item Delete the vertices of $\Tilde G_{b, centre}$ participating in $P_1$ and $P_5$. There are only two of them - the first and the last black vertices of $P$.
    \item Delete the edges of $\Tilde G_{b, centre}$ crossing an edge in $E(P)\cap E_w$. Again, there are only two of them - these are the diagonal edges crossing the second and the second-to-last edges of $P$.
\end{enumerate} 
To summarise, the graph $G_{b, centre}$ is obtained from the graph  $\Tilde G_{b, centre}$  by deleting edges and vertices as specified in 2 and 3. 

\begin{observation}\label{middle}
The graph $G_{b, centre}$ is connected.
\end{observation}
\begin{proof}
Notice that the restriction of $G_{b, centre}$ to $C_{[n]}$ has exactly two connected components, one of which consists of the vertex $(n, 0, 0)$ only, and the restriction of $G_{b, centre}$ to $C_{[n+1]}$ is a connected graph. Now, it remains to see that the edges $(n, 0, 0)(n+1, 1, 0)$ and $(n+1, 1, 0)(n, 1, 1)$ are present in $G_{b, centre}$.
\end{proof}

\begin{lemma}\label{spanning_tree_extend}
Let $P$ be a spine. There is a set of diagonal edges extending $P$ to a tree $T'$ 
containing all vertices of $C$ with the following properties:
\begin{itemize}
    \item $T'$ uses only diagonal edges except two edges of $C$, one in the starting segment and one in 
the ending segment of the spine.
    \item Every fundamental cycle of $T'$ as a spanning tree of $(V, E\cup E(T'))$ 
contains at least one of the paths $P_2$ from $A$ to $B$ or $P_4$ from $A'$ to $B'$ in $P$. In particular:
\begin{itemize}
    \item If $xy$ is an edge in $E\backslash E(T')$ with white vertex $x$ in $C_{[n+1, 2n+1]}$, then the fundamental cycle of the edge $xy$ in $T'$ contains the subpath $P_4$ of $P$.
    \item If $xy$ is an edge in $E\backslash E(T')$ with white vertex $x$ in $C_{[0, n]}$, then the fundamental cycle of the edge $xy$ in $T'$ contains the subpath $P_2$ of $P$.
\end{itemize}
\end{itemize}
\end{lemma}
\begin{proof}
Like in the proof of Lemma \ref{spine}, we denote by $P_1$ the 
starting segment of $P$, by $P_3$ the path in $P$ from $B$ to $A'$ and by $P_5$ the ending segment of $P$. \par 

The graph $G_{b, right}$ is the induced subgraph of the graph $(V_b, E_b)$ with vertex set $V_b\cap C_{[n+2, 2n+1]}$. The graph $G_{b, right}$ is connected and contains the path $P_2$.  By Observation \ref{ob 3.3} the path $P_2$ can be extended to a spanning tree of $G_{b, right}$.  We choose one such spanning tree and denote it by $T^b_1$.\par

Similarly the graph $G_{b, left}$ is the induced subgraph of the graph $(V_b, E_b)$ with vertex set $V_b\cap C_{[0, n-1]}$. The graph $G_{b, left}$ is connected and contains the path $P_4$. Again, by Observation \ref{ob 3.3} the path $P_4$ can be extended to a spanning tree of $G_{b, left}$. We choose one such spanning tree and denote it by $T^b_2$.\par

The black vertices of $C$ not covered by $P$, $T^b_1$ and $T^b_2$ are the ones of $G_{b, centre}$. The graph $G_{b, centre}$ is connected by Observation \ref{middle}. We apply Observation \ref{ob 3.3} to the forest consisting of the second diagonal edge $e$ of the path $P_3$. Note that this forest is included in $G_{b, centre}$. We conclude that there is a spanning tree of $G_{b, centre}$ containing $e$. Choose one such spanning tree and denote it by $T^b_3$. Thus, the restriction of $P\cup T^b_1\cup T^b_2\cup 
T^b_3$ to $(V_b, E_b)$ forms a spanning tree of $(V_b, E_b)$, which we call $T^b$. (Indeed, it is connected as the set $P$ interests all the other three trees in the union and the union is acyclic and contains all black vertices by construction.)\par

Let $I$ be the set of diagonal edges with two white 
endvertices in $C$ intersecting an edge of $T^b$. As $T^b$ is a tree, the induced subgraph of $T_b$ obtained by restricting to the vertex set of $C_{[n+1, 2n+1]}$ is a forest. We apply Lemma \ref{connected} with $C_1 = 
C_{[n+1, 2n+1]}$ and $I_1$ the subset of $I$ consisting of those edges with both endvertices in $C_{[n+1, 2n+1]}$ to deduce that the induced subgraph of the graph $(V_w, E_w\backslash I)$ obtained by restricting to the vertex set of $C_{[n+1, 2n+1]}$ forms a connected graph that we call $G_{w, right}$.
By Observation \ref{ob 3.3} there is a spanning tree of $G_{w, right}$, which contains the last two diagonal edges of the ending segment $P_5$ of the spine $P$. We choose one such tree and call it $T^w_1$.\par

Similarly, as $T^b$ is a tree, the induced subgraph of $T_b$ obtained by restricting to the vertex set of $C_{[0, n]}$ is a forest. We apply Lemma \ref{connected} with $C_1 = 
C_{[0, n]}$ and $I_1$ the subset of $I$ with both endvertices in $C_{[0, n]}$ to deduce that the induced subgraph of the graph $(V_w, E_w\backslash I)$ obtained by restricting to the vertex set of $C_{[0, n]}$ forms a connected graph. We call that connected graph $G_{w, left}$.
By Observation \ref{ob 3.3} there is a spanning tree of $G_{w, left}$, which contains the first two diagonal edges of the starting segment $P_1$ of the spine $P$. We choose one such tree and call it $T^w_2$.\par

We define $T' = P\cup T^b\cup T^w_1\cup T^w_2$. We denote its vertex set by $V$ and its edge set by $E(
T')$. $T'$ is a tree, and hence a spanning tree of the graph $(V, E\cup E(T'))$. We now prove that every 
fundamental cycle of $T'$ contains at least one of the paths $P_2$ from $A$ to $B$ and 
$P_4$ from $A'$ to $B'$ in the spine $P$.\par
All of the edges in $E\backslash E(T')$ have one white and one black endvertex. We treat 
edges 
with white endvertex in $C_{[n+1, 2n+1]}$ and edges with white endvertex in $C_{[0,n]}$ separately. Choose an edge $xy$ in $E\backslash E(T')$ with white endvertex $x$. If $x$ is a vertex of $C_{[n+1, 2n+1]}$, then $x$ is a vertex of $T^w_1$. This means that $y$ has abscissa at least $n$ and is a vertex of one of the graphs
$P_1$, $P_2$, $P_3$, $T^b_1$ or $T^b_3$. Thus the fundamental cycle of the edge $xy$ in $T'$ contains $P_4$ by construction. 

Similarly, if $x$ is a vertex of $C_{[0,n]}$ and is consequently covered by $T^w_2$, then $y$ must belong to one the graphs $P_3$, $P_4$, 
$P_5$, $T^b_2$ or $T^b_3$. It follows that the fundamental cycle of the edge $xy$ in $T'$ contains $P_2$ by construction, which finishes the proof.
\end{proof}

\begin{figure}
\centering
\includegraphics[scale=0.7]{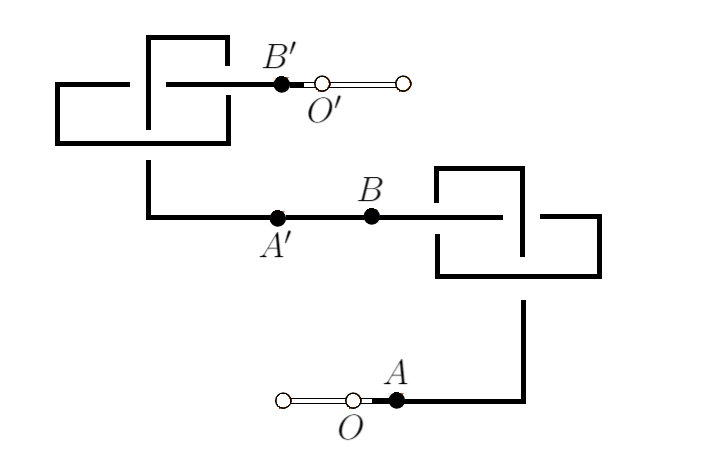}
\caption{An approximative scheme of a spine.}
\label{Kn}
\end{figure}

We now subdivide some of the faces of $C$ by using the edges of $T'$ with endvertices in the same colour. This defines the $2-$complex $C' = (V', E', F')$. 

As subdivisions of faces do not change the topological properties of the $2-$complex, $C'$ 
is a simply connected $2-$complex. Let us call the embedding of $C'$ in $3-$space obtained after 
subdivisions of faces of the canonical embedding of $C$ \textit{canonical embedding of 
$C'$}.

In the following lemma we prove that every fundamental cycle of $T'$ as a spanning tree of the $1-$skeleton of
$C'$ forms a nontrivial knot in the canonical embedding of $C'$. Otherwise said, we prove that $T'$ is entangled with respect to the canonical 
embedding of $C'$.

\begin{lemma}\label{L 3.6}
Every fundamental cycle of the spanning tree $T'$ forms a nontrivial knot in the 
canonical 
embedding of $C'$. 
\end{lemma}

\begin{proof}
All of the edges of $C'$ not in $T'$ have one white and one black endvertex. 
We treat edges 
with white endvertex with abscissa at least $n+1$ and edges with white endvertex with abscissa at 
most 
$n$ separately.\par 
Let $e = xy$ be an edge of $C'$ not in $T'$ with white endvertex $x$. If $x$ has abscissa at least $n+1$, then the fundamental cycle $o_e$ of $e$ contains the path $P_4$ by Lemma \ref{spanning_tree_extend}. Thus, we 
can decompose the knot formed by the embedding of the fundamental cycle 
$o_e$ induced by the canonical embedding of $C'$ as a connected sum of the knot $K$, containing $e$, the line segment 
between
$A'$ and $B'$ and the paths in $T'$ between $y$ and $A'$ and between $B'$ and $x$, and the knot
$K'$, 
containing only the line segment between $A'$ and $B'$ and $P_4$. See \autoref{k1k2}. As $K'$ is a 
nontrivial knot, the connected sum $K \# K'$ is a nontrivial knot by Lemma \ref{L 2.6}. This proves 
that the present embedding of $o_e$ forms a nontrivial knot.\par 
In the case when $x$ has abscissa at most $n$, the fundamental cycle 
$o_e$ of $e$ contains the path $P_2$ by Lemma \ref{spanning_tree_extend}, so its embedding, induced by the canonical embedding of $C'$, can be decomposed 
in a similar fashion as a connected sum of the knot $K$, containing $e$, the line segment between $A$ and 
$B$ and the paths in $T'$ between $x$ and $A$ and between $B$ and $y$, 
and the knot $K'$, containing only the line segment between $A$ and $B$ and $P_2$. Once again by Lemma \ref{L 2.6} $K \# K'$ is a nontrivial knot because $K'$ is a nontrivial knot. Thus $T'$ is 
entangled 
with respect to the canonical embedding of $C'$.
\end{proof}

\begin{figure}
\centering
\includegraphics[scale=0.7]{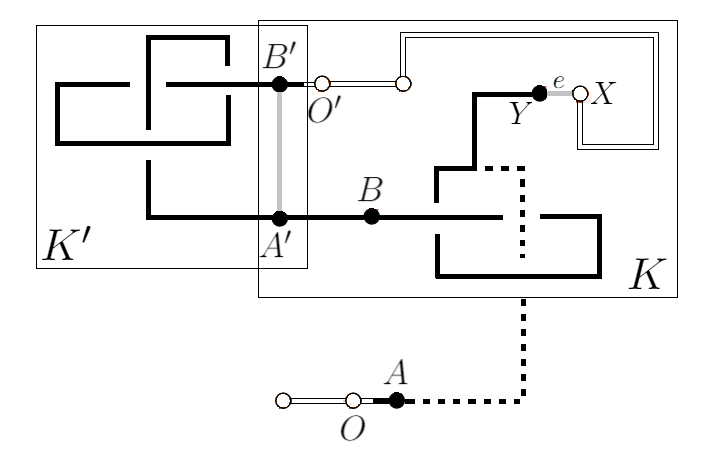}
\caption{$\gamma_e$ is a connected sum of $K$ and $K'$.}
\label{k1k2}
\end{figure}

We continue with the proof of \autoref{Thm 2}. Our next goal will be to prove the following lemma:
\begin{lemma}\label{embedC'}
The $2-$complex $C'$ has a unique embedding in $3-$space up to homeomorphism.
\end{lemma}

As the $2-$complex $C'$ is obtained from the cuboid complex $C$ by subdividing some of the faces of $C$, the two complexes are topologically equivalent. Therefore in the sequel we work with $C$ rather than $C'$ to avoid technicalities that have to do with the diagonal edges, which are irrelevant for the proof of Lemma \ref{embedC'}.

From (\citep{JC1}, Section 4) combined with Lemma \ref{iso} we know that every simply connected and locally $3-$connected\footnote{For every $k\geq 2$, a simplicial complex is \textit{locally $k-$connected} if each of its link graphs is $k-$connected.} simplicial complex embeddable in $\mathbb S^3$ has a unique embedding in $3-$space up to homeomorphism. One may be tempted to apply this result to the simply connected $2-$complex $C$ directly. Although the link graphs at most of its vertices are $3-$connected, this does not hold for all of them. For example, the link graph at the
vertex with coordinates $(1,0,0)$ in the canonical embedding of $C$ is equal to the complete graph $K_4$ minus an edge. It is easy to see that this graph can be disconnected by deleting the two vertices of degree 3. Another obstacle comes from the link graphs at the "corner vertices" of $C$ (take $(0,0,0)$ for example), which are equal to $K_3$ and are therefore only $2-$connected.

Our goal now will be to construct a $2-$complex, which contains $C$ as a subcomplex and is moreover embeddable in $3-$space, simply connected and locally $3-$connected at the same time. Roughly speaking, the construction consists of packing $C$ (seen in its canonical embedding) with one layer of unit cubes to obtain a cuboid complex of size $(2n+3)\times (n+2)\times (n+2)$ containing $C$ in its inside, and then contract all edges and faces disjoint from $C$.

The formal construction goes as follows, see Figure \ref{3-connFIG}. Let $C^+$ be the cuboid complex of size $(2n+3)\times (n+2)\times (n+2)$. Let $\iota^+$ be its canonical embedding. The restriction of $\iota^+$ to the cuboid $[1,2n+2]\times [1,n+1]\times [1,n+1]$ is the canonical embedding of $C$ (translated to the vector $(1,1,1)$). Thus we view $C$ as a subcomplex of $C^+$.

\begin{observation}\label{Ob3}
The $2-$complex $C^+$ is simply connected.\qed
\end{observation}

Let us contract all edges and faces of $C^+$ disjoint from $C$ to a single vertex $t$. By Observations \ref{Ob2} and \ref{Ob3} this produces a simply connected $2-$complex $C^t$.

\begin{lemma}\label{3-conn}
The link graph at the vertex $t$ of the $2-$complex $C^t$ is $3-$connected.
\end{lemma}

\begin{proof}
Let us consider the embedding $\iota^t$ of the $2-$complex $C^t$ in $\mathbb S^3$ in which $\iota^t(t) = \infty$, $\iota^t_{|C = C^t\backslash \{t\}}$ is the canonical embedding of $C$ in $3-$space and for every face $f$ of $C^t$, $\iota^t(f)$ is included in some affine plane of $\mathbb R^3\cup \{\infty\}$. From this embedding of $C^t$ we deduce that the link graph at $t$ in $C^t$ can be embedded in $\mathbb R^3$ as follows. Consider the integer points (i.e. the points with three integer coordinates) on the boundary of the cuboid $\iota^t(C)$. Construct a copy of the $1-$skeleton of each side of $\iota^t(C)$ by translating it to an outgoing vector of length one orthogonal to this side. Then, add an edge between every pair of vertices, which are the images of the same integer point on the boundary of the cuboid $\iota^t(C)$ under two different translations. Otherwise said we add edges between the pairs of integer points in $\mathbb R^3$, which are in the copies of two different sides of the cuboid and at euclidean distance $\sqrt{2}$. See \autoref{3-connFIG}.

\begin{figure}
\centering
\includegraphics[scale=0.6]{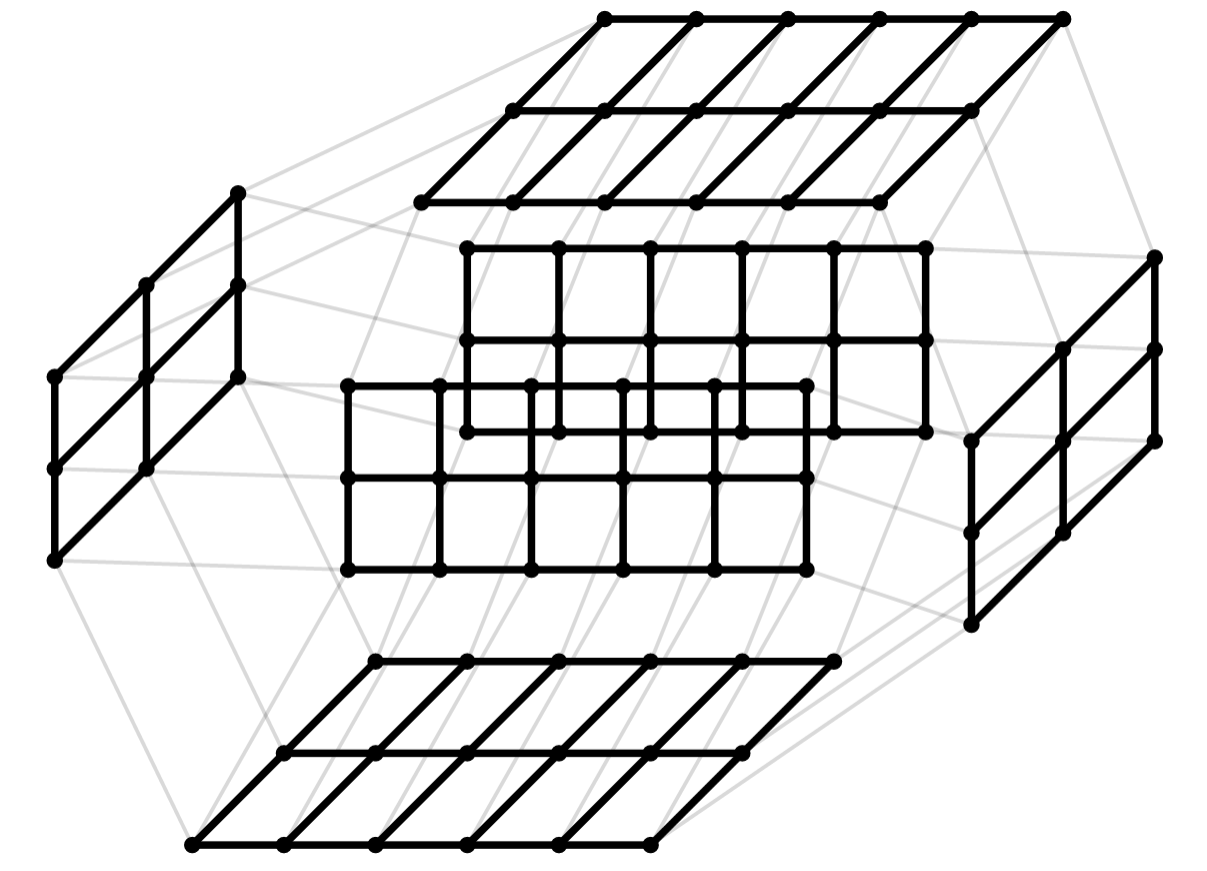}
\caption{The link graph at $t$ in $C^t$. Here $n=2$. The copies of all six sides are depicted in black while the edges added between between copies of two different sides are coloured in light grey.}
\label{3-connFIG}
\end{figure}

We easily verify now that in the graph constructed above there are at least three vertex-disjoint paths between every two vertices (indeed, there are always four such paths). By Menger's theorem the link graph at $t$ in $C^t$ is then $3-$connected. 
\end{proof}

The \textit{double wheel graph} is the graph on six vertices, which is the complement of a perfect matching. We denote it by $W^2$.

\begin{corollary}\label{locally3-conn}
The $2-$complex $C^t$ is locally $3-$connected.
\end{corollary}

\begin{proof}
The link graph at $t$ in $C^t$ is $3-$connected by Lemma \ref{3-conn}. The link graphs at all other vertices are all equal to the double wheel graph, which is $3-$connected as well, which proves the claim. 
\end{proof}

Now, by Observation \ref{Ob3}, Corollary \ref{locally3-conn}, Lemma \ref{iso} and (\citep{JC1}, Section 4) we deduce that $C^t$, just like any other simply connected and locally $3-$connected $2-$complex embeddable in $3-$space, has a unique embedding in $\mathbb S^3$ up to homeomorphism.

\begin{corollary}\label{uniqueC}
The $2-$complex $C$ has a unique embedding in $3-$space up to homeomorphism.
\end{corollary}
\begin{proof}
Let $\iota$ be an embedding of $C$ in $3-$space. Consider the subcomplex $C_1$ of $C$ induced by the vertices of $C$ with coordinates (taken with respect to the canonical embedding of $C$) in the set 

\begin{equation*}
    \Big\{(x, y, z)| \hspace{0.4em} x\in \{0, 2n+1\}\Big\}\bigcup  \Big\{(x, y, z)| \hspace{0.4em} y\in \{0, n\}\Big\}\bigcup  \Big\{(x, y, z)| \hspace{0.4em} z\in \{0, n\}\Big\}.
\end{equation*}

These are roughly the "boundary vertices" of $C$ in its canonical embedding. Thus $\iota(C_1)$ is a piecewise linear embedding of the $2-$sphere in $3-$space. Now notice that $\mathbb S^3\backslash \iota(C_1)$ has two connected components. Moreover, as $\iota(C)\backslash \iota(C_1)$ is connected, it must lie entirely in one of the two connected components of $\mathbb S^3\backslash \iota(C_1)$. Adding a vertex $t$ to the connected component disjoint from $\iota(C)$ allows us to construct an embedding of $C^t$ in $3-$space. However, this embedding is unique up to homeomorphism of $\mathbb S^3$. We deduce that $C$ also has a unique embedding in $3-$space up to homeomorphism of $\mathbb S^3$.
\end{proof}

We are ready to prove Lemma \ref{embedC'}.
\begin{proof}[Proof of Lemma \ref{embedC'}.]
Every embedding of the $2-$complex $C'$ comes from an embedding of $C$ by subdividing some of the faces of $C$ with the edges of $T'$. By Corollary \ref{uniqueC} there is a unique embedding of $C$ in $3-$space up to homeomorphism. Thus $C'$ has a unique embedding in $3-$space up to homeomorphism as well.
\end{proof}

Towards the proof of \autoref{Thm 2}, we prove the following lemma:
\begin{lemma}\label{main L}
Every cycle $o$ of $C'$ that is a nontrivial knot in the canonical embedding of $C'$ is a nontrivial knot in any embedding of $C'$.
\end{lemma}

First we need one more lemma and a corollary. 

\begin{lemma}\label{L 3.8}
Let $\psi: \mathbb S^3\longrightarrow \mathbb S^3$ be a homeomorphism of the $3-$sphere. Let 
$\gamma$ be a trivial knot in $\mathbb S^3$. Then the knot $\psi(\gamma)$ is trivial.
\end{lemma}
\begin{proof}
As $\gamma$ is a trivial knot, 
it has a thickening whose complement is 
homeomorphic to a solid torus. We call this thickeining $D$. By the Solid Torus Theorem (see 
\citep{Al} or \citep{JR}) the complement of $D$ -- that is, $\mathbb S^3\backslash D$ -- is a solid 
torus. As 
$\psi$ is a homeomorphism, the image $\psi(\gamma)$ of 
the knot $\gamma$ is a knot. 
By intersecting the thickening $D$ of $\gamma$ with the inverse image of a thickening of the knot 
$\psi(\gamma)$ if necessary, we may assume that additionally also $\psi(D)$ is a thicking of the 
knot $\psi(\gamma)$. 

The restriction of the homeomorphism $\psi$ to the knot complement $\mathbb S^3\backslash D$ 
is a homeomorphism to $\mathbb S^3\backslash \psi(D)$. Thus these two knot complements are 
homeomorphic. 
By the Gordon-Luecke Theorem \citep{GL}, it follows that the knots $\gamma$ and $\psi(\gamma)$ have 
the same knot type. Thus the knot $\psi(\gamma)$ must be trivial. 
\end{proof}

\begin{corollary}\label{cor 3.9}
The image of a nontrivial knot in $\mathbb S^3$ by a homeomorphism $\psi$ of the $3-$sphere is a 
nontrivial knot.
\end{corollary}
\begin{proof}
This is the contraposition of Lemma \ref{L 3.8} applied to $\psi^{-1}$.
\end{proof}

We are now ready to prove Lemma \ref{main L}.
\begin{proof}[Proof of Lemma \ref{main L}]
This is a direct consequence of Lemma \ref{embedC'} and Corollary \ref{cor 3.9}.
\end{proof}

We are now able to complete the proof of \autoref{Thm 2}. It remains to prove that the spanning tree $T'$ of the 
$1-$skeleton of $C'$ is entangled (recall that this means that each of its fundamental cycles forms a nontrivial knot in any embedding of $C'$ in $3-$space). 

\begin{proof}
Consider an edge $e$ in $E'\backslash E(T')$. By Lemma \ref{L 3.6} the fundamental cycle $o_e$ 
of $T'$ is nontrivially knotted in the canonical embedding of $C'$. By Lemma \ref{embedC'} any two embeddings of $C'$ in $3-$space are homeomorphic, so applying Lemma 
\ref{main L} to $o_e$ gives that $o_e$ forms a nontrivial knot in every embedding of $C'$ in 
$3-$space. As this holds for every edge in $E'\backslash E(T')$ the proof of \autoref{Thm 2} 
is complete.
\end{proof}

\section{Proof of Lemma \ref{sec4lemma}}\label{Section4}
Let us consider the $2-$complex $C'$ and the spanning tree $T'$ of the 
$1-$skeleton of $C'$ as in \autoref{Thm 2}. We recall that the $2-$complex $C'' = (V'', E'', F'')$ is obtained by contraction of the  spanning tree $T'$ of the $1-$skeleton of $C'$. Let us consider an embedding $\iota'$ of the 
$2-$complex $C'$ in $3-$space. By Observation \ref{Ob1} contractions of edges with different endvertices preserve 
embeddability and can be performed within $\iota'$. Therefore contracting the edges of the tree $T'$ one by one
within $\iota'$ induces an embedding $\iota''$ of $C''$ in which every 
edge forms a nontrivial knot. The goal of this section is to justify that every embedding of $C''$ in $3-$space can be obtained this way.\\

We recall that for a $2-$complex $C _1 = (V_1, E_1, F_1)$, the \textit{link graph} 
$L_v(C _1)$ at the vertex $v$ in $C_1$ is the incidence graph between edges and 
faces incident with $v$ in $C _1$.  

Below we aim to show that every planar rotation system of the $2-$complex $C''$ arises from 
a planar rotation system of the $2-$complex $C'$. We begin by proving that contractions of 
edges of a $2-$complex commute with each other.

\begin{lemma}\label{comm}
Let $e_1, e_2, \dots , e_k$ be edges of a $2-$complex $C_1$. The link graphs at the 
vertices of the $2-$complex $C_1/\{e_1, e_2, \dots e_k\}$ do not depend on the order in which 
the edges $e_1, e_2, \dots , e_k$ are contracted. 
\end{lemma} 
\begin{proof}
It is sufficient to observe that the $2-$complex $C_1/\{e_1, e_2,\dots, e_k\}$ is well defined 
and does not depend on the order of contraction of the edges $e_1, e_2, \dots, e_k$.
\end{proof}

\begin{lemma}\label{L 4.3}
Let $C _1 = (V_1, E_1, F_1)$ be a locally $2-$connected $2-$complex and let $e$ be an edge 
of $C _1$ that is not a loop. Then every planar rotation system of the $2-$complex 
$C _1/e$ is induced by a planar rotation system of $C _1$.
\end{lemma}
\begin{proof}
Let $e = xy$ for $x, y\in V_1$. As the link graphs at $x$ and $y$ are $2-$connected, the vertices 
corresponding to the edge $e$ in the two link graphs $L_x(C _1)$ and $L_y(C _1)$ 
are not cutvertices. Under these conditions (\citep{JC1}, Lemma 2.2) says that every planar
rotation system of $C _1/e$ is induced by a planar rotation system of $C _1$.
\end{proof}

\begin{observation}\label{2-conn'}
Subdivisions of $2-$connected graphs are $2-$connected.
\end{observation}
\begin{proof}
Let $G$ be a $2-$connected graph and $G'$ be a subdivision of $G$. Let $v'$ be a vertex of $G'$. If the vertex $v'$ is present in $G$, then $G\backslash v'$ can be obtained from $G'\backslash v'$ by a sequence of edge contractions, so in particular $G'\backslash v'$ is connected. If the vertex $v'$ is not present in $G$ and participates in the subdivision of the edge $e$ of $G$, then $G\backslash e$ can be obtained from $G'\backslash v'$ by a sequence of edge contractions, so $G'\backslash v'$ is connected. 
\end{proof}

We now state and prove an easy but crucial observation.
\begin{observation}\label{2-conn}
The $2-$complexes $C$ and $C'$ are locally $2-$connected.
\end{observation}
\begin{proof}
As the link graphs at the vertices of $C'$ are subdivisions of the link graphs at the vertices of $C$ (to construct $C'$ we only add new edges subdividing already existing faces of $C$), by Observation \ref{2-conn'} it is sufficient to prove the observation for the $2-$complex $C$.\par
By \textit{degree of a vertex $v$} in $C$ we mean the number of edges of $C$ incident to $v$. The link graphs at the vertex $v$ of $C$ are equal to:
\begin{itemize}
    \item The double wheel graph $W^2$ if $v$ is of degree 6.
    \item $W^2\backslash w$, where $w$ is any vertex of $W^2$, if  $v$ is of degree 5.
    \item $K_4\backslash e$, where $e$ is any edge of the complete graph $K_4$, if  $v$ is of degree 4.
    \item The complete graph $K_3$, if $v$ is of degree 3.
\end{itemize}
As each of these graphs is $2-$connected, the $2-$complex $C$ is locally $2-$connected.
\end{proof}

\begin{corollary}\label{main cor}
Every planar rotation system of $C''$ is induced by a planar rotation system of $C'$.
\end{corollary}
\begin{proof}
As contractions of edges commute 
by Lemma \ref{comm}, the order of contraction of the edges of the tree $T'$ is irrelevant.\par
We know that the $2-$complex $C'$ is locally 
$2-$connected and by (\citep{JC1}, Lemma 3.4) we also know that vertex sums of $2-$connected graphs are 
$2-$connected. From these two facts we deduce that the assumptions of 
Lemma \ref{L 4.3} remain satisfied after each contraction. Thus we use Lemma \ref{L 4.3} inductively by performing consecutive contractions of the edges of the 
spanning tree $T'$ of the $1-$skeleton of $C'$, which proves the corollary.
\end{proof}

\begin{lemma}\label{iso}
Let $\iota$ and $\iota'$ be two embeddings of a locally connected and simply connected $2-$complex in $3-$space 
with the same planar rotation systems. Then there is a homeomorphism $\psi$ of the $3-$sphere such 
that the concatenation of $\iota$ and $\psi$ is $\iota'$.\footnote{A consequence of this lemma is that simply connected locally 3-connected 2-complexes have unique embeddings in 3-space. This was observed independently by Georgakopoulos and Kim.} \end{lemma}
\begin{proof}
Consider thickenings\footnote{A \textit{thickening $D$} of an embedding $\iota$ of a $2-$complex in 
$3-$space is the manifold $\iota + B(0, \varepsilon)$ for $\varepsilon > 0$ such that the number of 
connected components of $\mathbb S^3 \backslash \iota$ is equal to the number of connected 
components of $\mathbb S^3 \backslash D$. Here $B(0, \varepsilon)$ is the closed $3-$ball of center 
0 and radius $\varepsilon$.} $D$ and $D'$ of the embeddings $\iota$ and $\iota '$. As these 
embeddings are assumed to be piecewise linear, $D$ and $D'$ are well defined up to homeomorphism. 
Moreover, as the planar rotation systems of $\iota$ and $\iota '$ coincide, $D$ and $D'$ are 
homeomorphic. We denote the homeomorphism between $D$ and $D'$ by $\psi$. Firstly, as the image of 
the boundary of $D$ under $\psi$ is the boundary of $D'$, $\psi$ induces a bijection between the 
connected components of $\mathbb S^3 \backslash D$ and the connected components of $\mathbb S^3 
\backslash D'$. More precisely, the connected component $B$ of $\mathbb S^3 \backslash D$ 
corresponds to the connected component $B'$ of $\mathbb S^3 \backslash D'$ for which $\psi (\partial 
B) = \partial B'$. Secondly, as the $2-$complex $C$ is simply connected and locally 
connected, all connected componnets of $\mathbb S^3 \backslash D$ and of $\mathbb S^3 \backslash D'$ 
have boundaries homeomorphic to the $2-$sphere. See for example Theorem 6.8 in \citep{JC2}. By 
Alexander's Theorem every connected component is homeomorphic to the $3-$ball.\par
Fix a pair $(B, B')$ as above. By a trivial induction argument it is sufficient to extend $\psi$ 
from $D\cup B$ to $D'\cup B'$. By performing isotopy if necessary, we have that $B$ and $B'$ are 
convex. Choosing some $b\in B$ and $b'\in B'$, we construct a homeomorphism $\overline \psi :B 
\longrightarrow B'$ as $\forall \lambda \in [0,1), \forall x\in \partial B, \overline \psi (b + 
\lambda (x-b)) = b' + \lambda (\psi(x) - b')$. Thus, $\psi \cup \overline \psi$ gives the required 
homeomorphism from $D\cup B$ to $D'\cup B'$.
\end{proof}

We are ready to prove Lemma \ref{sec4lemma} saying that every embedding of $C''$ in $3-$space is obtained from an embedding of $C'$ by contracting the tree $T'$.

\begin{proof}[Proof of Lemma \ref{sec4lemma}]
Consider an embedding $\iota''$ of the $2-$complex $C''$ in $3-$space with planar rotation system $\Sigma''$. By Corollary \ref{main cor} $\Sigma''$ is induced by a planar rotation system $\Sigma'$ of $C'$. As the $2-$complex $C'$ is simply connected and has a planar rotation system $\Sigma'$, by (\citep{JC2}, Theorem 1.1) it has an embedding $\iota'$ in $3-$space with rotation system $\Sigma'$. Contraction of the tree $T'$ in the $2-$complex $C'$ produces an embedding of $C''$ with planar rotation system $\Sigma''$, which is homeomorphic to $\iota''$ by Lemma \ref{iso}. This proves Lemma \ref{sec4lemma}.
\end{proof}

We conclude this section with two consequences of Lemma \ref{sec4lemma}.

\begin{corollary}\label{cor 4.5}
The $2-$complex $C''$ has a unique embedding in $3-$space up to homeomorphism.
\end{corollary}
\begin{proof}
By Lemma \ref{embedC'} there is a unique embedding of $C'$ in $\mathbb S^3$ up to homeomorphism. By Lemma \ref{sec4lemma} we conclude that there is a unique embedding of $C''$ in $\mathbb S^3$ up to homeomorphism as well.
\end{proof}

\begin{corollary}\label{cor 4.6}
Every embedding of the $2-$complex $C''$ in $3-$space contains only edges forming nontrivial knots.
\end{corollary}
\begin{proof}
Let $\iota''$ be an embedding of $C''$. By Lemma \ref{sec4lemma} there is an embedding $\iota'$ of $C'$ in $3-$space, which induces $\iota''$. Let $e''$ be an edge of $C''$. It corresponds to an edge $e'$ of $C'$, which is not in $T'$. As the tree $T'$ is entangled, the embedding of the fundamental cycle of $e'$ in $T'$ induced by $\iota'$ forms a nontrivial knot. Remains to notice that this knot must have the same knot type as $\iota''(e'')$. Thus for every embedding $\iota''$ of $C''$ in $3-$space and every edge $e''$ of $C''$ we have that $\iota''(e'')$ is a nontrivial knot.
\end{proof}

\section{Proof of Lemma \ref{sec5lemma}}\label{Section5}

The remainder of this paper is dedicated to the proof of Lemma \ref{sec5lemma}, which will be implied by the following lemma.

\begin{lemma}\label{rem_lem}
For every edge $e''$ of $C''$ the link graph of $C''/e''$ at its unique vertex is not planar.
\end{lemma}

\begin{proof}[Proof that Lemma \ref{rem_lem} implies Lemma \ref{sec5lemma}]
Consider the 2-complex $C''/e''$ for some edge $e''$ of $C''$.
By Lemma \ref{rem_lem}, the link graph at its unique vertex is not planar. Hence $C''/e''$ is not embeddable in any 3-manifold.
\end{proof}

Before proving Lemma \ref{rem_lem}, we do some preparation. \par

\begin{lemma}\label{L 5.2}
Let the graph $G$ be a vertex sum of the two disjoint graphs $G'$ and $G''$ at the vertices $x'$ and $x''$, respectively. Suppose that $G'$ is not planar and $G''$ is $2-$connected. Then, $G$ is not planar. 
\end{lemma}
\begin{proof} 
As the graph $G''$ is $2-$connected, the graph
$G''\backslash x''$ is connected. Therefore by contracting the graph $G$ onto the edge set of $G'$,
we obtain the graph $G'$ (notice that contraction of a loop edge is equivalent to its deletion). 
As contraction of edges preserves planarity, if $G'$ is not planar, then $G$ is not planar as well.
\end{proof}

For a $2-$complex $C_1$ and edges $e_1, e_2, \dots, e_k$ in $C_1$ there is a 
bijection between the edges of $C_1$ different from $e_1, e_2, \dots, e_k$ and the edges of 
$C_1/\{e_1, e_2, \dots, e_k\}$. In order to increase readability, we suppress this bijection in out notation below; that is, we identify an edge $e$ of $
C_1$ different from $e_1, e_2, \dots, e_k$ with its corresponding edge of $C_1/\{e_1, e_2, \dots, e_k\}$.\par 

Let $e$ be an edge of a $2-$complex $C_1$. We aim to see how the link graphs at the 
vertices of $C_1$ relate to the link graphs at the vertices of $C_1/e$. Clearly 
link graphs at vertices not incident with the edge $e$ remain unchanged. If $e = uv$ for different 
vertices $u$ and $v$ of $C_1$, then contracting the edge $e$ leads to a vertex sum of the 
link graph at $u$ and the link graph at $v$ at the vertices $x$ and $y$ corresponding to the edge $e$. The bijection 
between their incident edges $(xx_i)_{i\leq k}$ and $(yy_i)_{i\leq k}$ is given as follows. The edge 
$xx_i$ in the link graph at $u$ corresponds to the edge $yy_i$ in the link graph at $v$ if both 
$xx_i$ and $yy_i$ are induced by the same face of $C_1$ incident to $e$.
If the edge $e$ is 
a loop with base vertex \footnote{A \textit{base vertex} of a loop edge is the only vertex this edge 
is incident with.} $v$ (i.e. $e = vv$), the link graph $L_v$ at $v$ is modified by the contraction of $e$ as follows. 

Let $x$ and $y$ be the vertices of $L_v$ corresponding to the loop edge $e$. 
Firstly, delete all edges between $x$ and $y$ in $L_v$. These edges correspond to the faces of $C_1$ having only the edge $e$ on their boundary.
Secondly, for every pair $(xx', yy')$ of edges of $L_v$ incident to the same face of $C_1$, add an edge between $x'$ and $y'$ in $L_v$. This edge might be a loop if $x'$ and $y'$ coincide.
Finally, delete the vertices $x$ and $y$ from $L_v$.\par
We call the graph obtained by the above sequence of three operations on the link graph $L_v$ \textit{internal vertex sum within the link graph $L_v$ at the vertices $x$ and $y$}. By abuse of language we also use the term \textit{internal vertex sum} for the sequence of operations itself.\par

\begin{lemma}\label{prove 5.1}
Let $o$ be a fundamental cycle of the spanning tree $T'$ of the $1-$skeleton of the $2-$complex $C'$. Contract the cycle $o$ to a vertex $\underline{o}$. Then, the link graph at the vertex $\underline{o}$ in the $2-$complex $C'/o$ is nonplanar.
\end{lemma}

Before proving Lemma \ref{prove 5.1} we show how Lemma \ref{L 5.2}, Lemma \ref{prove 5.1} and some results from previous sections together imply Lemma \ref{rem_lem}.

\begin{proof}[Proof that Lemma \ref{prove 5.1} implies  Lemma \ref{rem_lem}.]
Let $e''$ be an edge of the $2-$complex $C''$. It originates from an edge $e'$ of $C'$, which is not in $T'$. Thus, $e'$ participates in a fundamental cycle $o$ of $T'$. As contractions of edges of a $2-$complex commute by Lemma \ref{comm}, we obtain $C''/e''$ by first contracting the edges of $o$ in $C'$ and then the edges of $T'$ not in $o$ in $C'/o$. By Lemma \ref{prove 5.1} contracting $o$ to a vertex $\underline{o}$ in $C'/o$ leads to a nonplanar link graph at $\underline{o}$. Moreover, as the $2-$complex $C'$ is locally $2-$connected by Observation \ref{2-conn}, the link graph at every vertex of $C'/o$ except possibly $\underline{o}$ is $2-$connected. Then, by Lemma \ref{L 5.2} contraction of any non-loop edge $e = \underline{o}w$ of $C'/o$ incident to $\underline{o}$ leads to a non-planar link graph at the vertex of $C'/\{o,e\}$ obtained by identifying $\underline{o}$ and $w$. Then, contracting one by one the edges of $E(T')\backslash E(o)$ in $C'/o$ to the vertex $\underline{o}$ and applying consecutively Lemma \ref{L 5.2} we deduce that the link graph at the only vertex of $C''/e''$ is not planar. (Here by abuse of notation we denote by $\underline{o}$ the vertex at which the link graph is not planar after each following contraction. In this sense $\underline{o}$ is also the only remaining vertex in $C''/e''$.)
\end{proof}

The aim of this section from now on will be to prove Lemma \ref{prove 5.1}.\par

Let $G_{14}$ be the graph depicted on the left of \autoref{New}. Formally its vertex set is
\begin{equation*}
    V(G_{14}) = \{X_1, X_2, X_3, Y_1, Y_2, Y_{3,1}, Y_{3,2}, K, L, M, N, Q, R, S\}
\end{equation*}
and its edge set is 
\begin{align*}
    E(G_{14}) = 
    &\{X_1Y_1, X_1Y_2, X_2Y_1, X_2Y_2, X_3Y_1, X_3Y_2, X_1Y_{3,1}, X_3K, KY_{3,1}, X_2L, LM, 
MY_{3,2},\\
    & Y_2K, Y_1K, X_1X_2, X_3S, LQ, LN, MQ, MN, RY_{3,2}, RQ, RN, RS, SQ, SN\}.
\end{align*}
We construct the graph $G_{13}$ from $G_{14}$ by identifying the vertices $Y_{3,1}$ and $Y_{3,2}$; the resulting identification vertex is denoted by $Y_3$.  See the right part of \autoref{New}.

\begin{lemma}\label{nonpl}
The graph $G_{13}$ is not planar.
\end{lemma}
\begin{proof}
We contract in the graph $G_{13}$ the paths $X_2LMY_3$ and $X_3KY_3$ each to a single edge.
The resulting graph contains all edges between the two vertex sets $\{X_1, X_2, X_3\}$ and $\{Y_1, Y_2, Y_3\}$.
So $G_{13}$ has $K_{3,3}$ as a minor. So $G_{13}$ cannot be planar as it has a nonplanar minor. 
\end{proof}

We make two essential reminders. Firstly, consider the canonical embedding of $C'$. The paths 
$P_2$ and $P_4$ are constructed so that there is a sequence of three consecutive diagonal edges 
pointing in the same direction. For example in $P_2$ as given in \autoref{FH} a possible choice of 
such sequence is the third, the fourth and the fifth edge after the vertex $A$. Secondly, every 
fundamental cycle obtained by adding an edge in $E'\backslash T'$ to $T'$ contains 
at least one of the paths $P_2$ and $P_4$ as a subpath by construction. Thus, fixing a fundamental 
cycle $o$ in $T'$, we find a path $e_1, e_2, e_3$ of three consecutive diagonal edges in 
the  same direction. We denote the four vertices in this path of three edges $e^-_1, e^+_1\equiv e^-_2, 
e^+_2\equiv e^-_3$ and $e^+_3$.\par

\begin{observation}\label{Ob 5.3}
The link graph at the vertex $e^+_2$ of $C'/e_2$ (where $e^+_1\equiv e^-_2\equiv e^+_2\equiv e^-_3$ in $C'/e_2$) is equal to $G_{14}$.
\qed
\end{observation}

Recall that the double wheel graph $W^2$ is a graph on six vertices, which is the complement of a perfect matching. Notice that for every edge $e$ of $W^2$ the graph $W^2\backslash e$ is the same. We call this graph \textit{modified double wheel graph} and denote it by $W^{2-}$.

\begin{observation}\label{Ob4}
Subdivisions of the double wheel graph $W^2$ and of the modified double wheel graph $W^{2-}$ are $2-$connected.
\qed
\end{observation}

\begin{lemma}\label{second to last}
Let the $2-$complex $C^- = C'\backslash \{e_1, e_3\}$ be obtained from the $2-$complex $C'$ by deleting the edges $e_1$ and $e_3$. Contract the path $p$ between $e^+_3$ and $e^-_1$ in $C^-$ contained in $o$ to a single vertex. The link graph obtained at this vertex after the contraction of $p$ is $2-$connected.
\end{lemma}
\begin{proof}
Fix a vertex $s$ of $C^-$ in $p$. If $s$ is different from $e^-_1$ and $e^+_3$, the link graph at $s$ in $C^-$ is equal to the link graph at $s$ in $C'$, which is a subdivision of $W^2$. By Observation \ref{Ob4} this graph is $2-$connected. If $s$ is equal to $e^-_1$ or $e^+_3$, then the link graph at $s$ in $C^-$ is a subdivision of the modified double wheel graph, which is again $2-$connected by Observation \ref{Ob4}. By (\citep{JC1}, Lemma 3.4) vertex sums of $2-$connected graphs are 
$2-$connected, which proves the lemma.
\end{proof}

The argument behind the next proof, despite being a bit technical, is quite straightforward. Informally it states that by plugging certain graphs $L_w$ into the graph $G_{14}$ 
twice via "vertex sums" at the vertices $Y_{3,1}$ and $Y_{3,2}$ of $G_{14}$ we obtain a graph containing $G_{13}$ as a minor.

\begin{figure}
\centering
\includegraphics[scale=0.5]{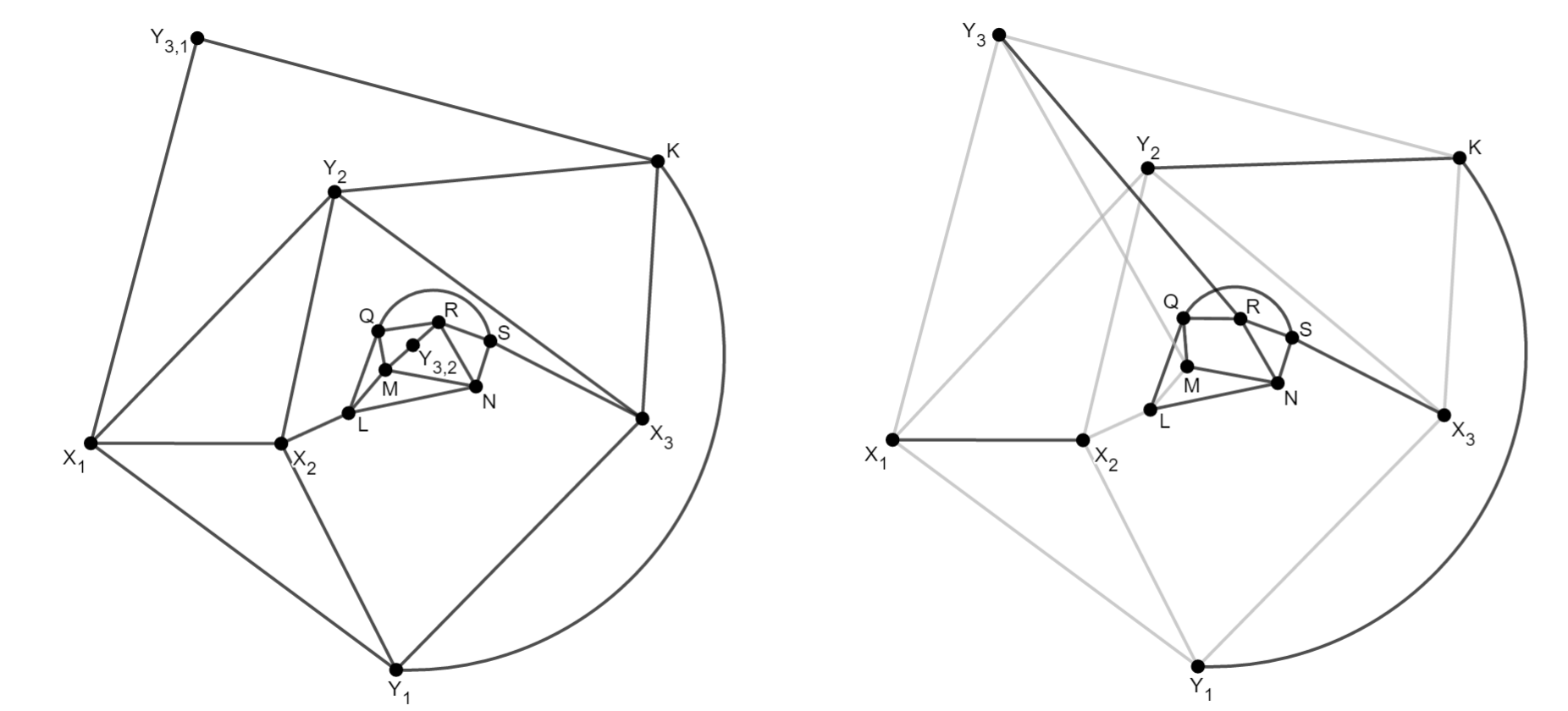}
\caption{The graph $G_{14}$ depicted on the left is obtained as link graph at the vertex $e^+_2$ after contraction of the 
edge $e_2$ in $C'$. After identification of $Y_{3, 1}$ and $Y_{3, 2}$ in $G_{14}$ we 
obtain the graph $G_{13}$ shown on the right. The subdivision of $K_{3,3}$ in $G_{13}$ is given in grey.}
\label{New}
\end{figure}

\begin{lemma}\label{minor}
Let $o$ be a fundamental cycle in $T'$. Contract the cycle $o$ to a vertex $\underline{o}$. 
Then, the link graph $L_{\underline{o}}$ at the vertex $\underline{o}$ in $C'/o$ has $G_{13}$ as a minor.
\end{lemma}

\begin{proof}
By Lemma \ref{comm} contractions of edges of a $2-$complex commute. Thus, we contract the edges of 
the cycle $o$ in the following order:
\begin{enumerate}
    \item We contract all edges except for $e_1, e_2, e_3$;
    \item we contract $e_2$, $e_1$ and $e_3$ in this order.
\end{enumerate}
We now follow in detail each of the described contractions. Let $L_w$ and $L_u$ be the link graphs at the vertices $w = e^{-}_1 = e^{+}_3$ and $u = e^{+}_2 = e^{-}_2$ respectively just before the contraction of the edge $e_1$ of $C'$. They are both $2-$connected as vertex sums of $2-$connected graphs. Let $Y'_{3,2}$ and $Y'_{3,1}$ correspond to the edges $e_3$ and $e_1$ respectively in the link graph $L_w$ at the vertex $w$. Analogously $Y_{3,2}$ and $Y_{3,1}$ correspond to the edges $e_3$ and $e_1$ respectively in the link graph $L_u$ at the vertex $u$, which is equal to $G_{14}$ by Observation \ref{Ob 5.3}. See \autoref{New}. Contractions of $e_1$ and $e_3$ produce the $2-$complex $C'/o$. The link graph $L_{\underline{o}}$ at the vertex $\underline{o}$ in $C'/o$ is obtained from $L_w$ and 
$L_u$ by performing:

\begin{itemize}
    \item A vertex sum between $L_w$ and $L_u$ at $Y'_{3,1}$ and $Y_{3,1}$ respectively. Call this vertex sum $L$.
    \item An internal vertex sum within $L$ at the vertices $Y'_{3,2}$ and $Y_{3,2}$.
\end{itemize}

The internal vertex sum within $L$ forms the link graph $L_{\underline{o}}$.\par
By Lemma \ref{second to last} the graph $L_{w}\backslash \{Y'_{3,1}, Y'_{3,2}\}$ is $2-$connected, so connected in particular. It is also 
realised as an induced subgraph of $L_{\underline{o}}$ by restricting $L_{\underline{o}}$ to the set of vertices inherited from $L_{w}$ (all except $Y'_{3,1}$ and $Y'_{3,2}$). The contraction of the edges of this induced subgraph within $L_{\underline{o}}$ is equivalent to identifying $Y_{3,1}$ and $Y_{3,2}$ in $L_u = G_{14}$. This proves the lemma.
\end{proof}

We are ready to prove Lemma \ref{prove 5.1}.
\begin{proof}[Proof of Lemma \ref{prove 5.1}]
By Lemma \ref{nonpl}, $G_{13}$ is not planar. At the same time, $G_{13}$ is a minor of the link graph $L_{\underline{o}}$ at the vertex $\underline{o}$ of $C'/o$ of by 
Lemma \ref{minor}. As contraction of edges preserves planarity, $L_{\underline{o}}$ is not planar as well.
\end{proof}

\section{Conclusion}\label{sec6}
In this paper we provided an example of a simply connected $2-$complex $C ''= (V'', E'', 
F'')$ embeddable in $3-$space such that the contraction of any edge $e$ of $C''$ in the 
abstract sense produces a $2-$complex $C ''/e$, which cannot be embedded in $3-$space. This 
construction opens a number of questions. Some of them are given below.\par

\begin{question}
Is there a structural characterisation of the (simply connected) $2-$complexes 
with exactly one vertex embeddable in $3-$space with the above property?
\end{question}

\begin{question}
Is there a structural characterisation of the (simply connected) $2-$complexes 
with exactly one vertex admitting an embedding in $3-$space without edges forming nontrivial 
knots?
\end{question}

\begin{question}
Is there a structural characterisation of the (simply connected) $2-$complexes 
such that each of their edge-contractions admits an embedding in $3-$space?
\end{question}

\section{Acknowledgements}
The second author would like to thank Nikolay Beluhov for a number of useful discussions.

\bibliographystyle{alpha}
\bibliography{Bibliography}

\begin{thebibliography}{ABL17}

\bibitem[ABL17]{MR3639606}
K.~Adiprasito, B.~Benedetti, and F.~Lutz.
\newblock Extremal examples of collapsible complexes and random discrete
  {M}orse theory.
\newblock {\em Discrete Comput. Geom.}, 57(4):824--853, 2017.

\bibitem[Ale24]{Al}
J.~W. Alexander.
\newblock On the subdivision of 3-space by a polyhedron.
\newblock {\em Proceedings of the National Academy of Sciences of the United
  States of America}, 10(1):6—8, January 1924.

\bibitem[Cara]{JC1}
J.~Carmesin.
\newblock Embedding simply connected 2-complexes in 3-space -- \uppercase{I. A
  K}uratowski-type characterisation.
\newblock Preprint 2017, available at "https://arxiv.org/pdf/1709.04642.pdf".

\bibitem[Carb]{JC2}
J.~Carmesin.
\newblock Embedding simply connected 2-complexes in 3-space -- \uppercase{II. A
  K}uratowski-type characterisation.
\newblock Preprint 2017, available at "https://arxiv.org/pdf/1709.04642.pdf".

\bibitem[Carc]{3space5}
J.~Carmesin.
\newblock Embedding simply connected 2-complexes in 3-space -- {V}. {A} refined
  {K}uratowski-type characterisation.
\newblock Preprint 2019, available at "https://arxiv.org/pdf/1709.04659.pdf".

\bibitem[GL89]{GL}
C.~Gordon and J.~Luecke.
\newblock Knots are determined by their complements.
\newblock {\em Bulletin of the American Mathematical Society}, 20(1), 1989.

\bibitem[Han16]{ZH}
Z.~Haney.
\newblock Lecture 12: Prime factorisation of knots, 2016.

\bibitem[JR69]{JR}
G.~Joubert and H.~Rosenberg.
\newblock Plongement du tore $\uppercase{T}^2$ dans la sphère
  $\uppercase{S}^3$ (in french).
\newblock {\em Cahiers Topologie Géom. Différentielle}, 11:323–328, 1969.

\bibitem[Kos92]{Kos}
A.~Kosinski.
\newblock {\em Differential Manifolds}.
\newblock Academic Press Inc, 1 edition, 1992.

\bibitem[Sch11]{Sch11}
S.~Schleimer.
\newblock Sphere recognition lies in {NP}.
\newblock {\em In Michael Usher, editor, Low- dimensional and Symplectic
  Topology. American Mathematical Society.}, 82:183--214, 2011.

\bibitem[Tho94]{MR1295555}
A.~Thompson.
\newblock Thin position and the recognition problem for {$S^3$}.
\newblock {\em Math. Res. Lett.}, 1(5):613--630, 1994.

\bibitem[Zen18]{Zen16}
R.~Zentner.
\newblock Integer homology 3-spheres admit irreducible representations in
  {${\rm SL}(2,\mathbb{C})$}.
\newblock {\em Duke Math. J.}, 167(9):1643--1712, 2018.

\end{thebibliography}

\end{document}